\documentclass[11pt]{article}
\usepackage{exscale,makeidx,amssymb,amsmath,
}
\usepackage{amsmath} 
\usepackage{amssymb} 

\usepackage{amsthm}

\usepackage{graphics}
\usepackage{graphicx}
 \DeclareGraphicsExtensions{.eps, .ps}

\newcommand{\QATOP}{}

\newtheorem{prop}{Proposition}
\newtheorem{defi}{Definition}
\newtheorem{lemm}[prop]{Lemma}
\newtheorem{theo}[prop]{Theorem}
\newtheorem{coro}[prop]{Corollary}

\newcommand{\alth}{{\alpha,\theta}}

\newcommand{\eq}{\begin{equation}}
\newcommand{\en}{\end{equation}}

\newcommand{\beq}{\begin{eqnarray*}}
\newcommand{\eeq}{\end{eqnarray*}}
\newcommand{\ds}{\displaystyle}

\def\build#1_#2^#3{\mathrel{\mathop{\kern 0pt#1}\limits_{#2}^{#3}}}
\newcommand{\beqs}{\begin{eqnarray*}&\ds}
\newcommand{\eeqs}{&\end{eqnarray*}}

\newcommand{\bA}{\mathbb{A}}
\newcommand{\bP}{\mathbb{P}}

\newcommand{\bN}{\mathbb{N}}
\newcommand{\bR}{\mathbb{R}}
\newcommand{\bD}{\mathbb{D}}

\newcommand{\bE}{\mathbb{E}}

\newcommand{\bT}{\mathbb{T}}

\newcommand{\cC}{\mathcal{C}}

\newcommand{\cL}{\mathcal{L}}

\newcommand{\cP}{\mathcal{P}}

\newcommand{\cR}{\mathcal{R}}
\newcommand{\cS}{\mathcal{S}}
\newcommand{\cT}{\mathcal{T}}

\newcommand{\ft}{\mathbf{t}}
\newcommand{\fs}{\mathbf{s}}

\begin{document}

\title{A new family of Markov branching trees: the alpha-gamma model}
\author{
Bo Chen\thanks{University of Oxford; email chen@stats.ox.ac.uk}
\and %
Daniel Ford\thanks{Google Inc.;
email dford@math.stanford.edu}
\and %
Matthias Winkel\thanks{%
University of Oxford; email winkel@stats.ox.ac.uk}}
\maketitle
\begin{abstract}
We introduce a simple tree growth process that gives
rise to a new two-parameter family of discrete fragmentation trees that extends Ford's
alpha model to multifurcating trees and includes the trees obtained by uniform
sampling from Duquesne and Le Gall's stable continuum random tree. We call
these new trees the alpha-gamma trees. In this paper, we obtain their
splitting rules, dislocation measures both in ranked order and in sized-biased
order, and we study their limiting behaviour.

\emph{AMS 2000 subject classifications: 60J80.\newline
}

\end{abstract}

\section{Introduction}

\em Markov branching trees \em were introduced by Aldous \cite{Ald-93}
as a class of random binary phylogenetic models and extended to the
multifurcating case in \cite{HMPW}. Consider the space $\bT_n$ of
combinatorial trees without degree-2 vertices, one degree-1 vertex
called the {\sc root} and exactly $n$ further degree-1 vertices
labelled by $[n]=\{1,\ldots,n\}$ and called the \em leaves\em; we
call the other vertices
\em branch points\em. Distributions 
on $\bT_n$ of random trees $T_n^*$ are determined by
distributions 
of the delabelled tree $T_n^\circ$ on the space $\bT_n^\circ$ of \em unlabelled
trees \em and conditional label distributions, e.g. \em exchangeable \em
labels. A sequence $(T_n^\circ,n\ge 1)$ of unlabelled trees has the \em
Markov branching property \em if for all $n\ge 2$ conditionally given that the
branching adjacent to the {\sc root} is into tree components whose numbers of
leaves are $n_1,\ldots,n_k$, these tree components are independent copies of
$T_{n_i}^\circ$, $1\le i\le k$. The distributions of the sizes in the first
branching of $T_n^\circ$, $n\ge 2$, are denoted by
\beqs q(n_1,\ldots,n_k),\qquad n_1\ge\ldots\ge n_k\ge 1,\quad k\ge 2:\quad n_1+\ldots+n_k=n,
\eeqs
and referred to as the \em splitting rule \em of $(T_n^\circ,n\ge 1)$.

Aldous \cite{Ald-93} studied in particular a one-parameter family
($\beta\ge-2$) that interpolates between several models known in
various biology and computer science contexts (e.g. $\beta=-2$ comb,
$\beta=-3/2$ uniform, $\beta=0$ Yule) and that he called the \em
beta-splitting model\em, he sets for $\beta>-2$:
\beq q^{{\rm Aldous}}_\beta(n-m,m)=\frac{1}{Z_n}{n\choose m} B(m+1+\beta,n-m+1+\beta),&&\quad \mbox{for $1\le m<n/2$,}\\
  q^{\rm Aldous}_\beta(n/2,n/2)=\frac{1}{2Z_n}{n\choose n/2}B(n/2+1+\beta,n/2+1+\beta),&&\quad\mbox{if $n$ even,}
\eeq
where $B(a,b)=\Gamma(a)\Gamma(b)/\Gamma(a+b)$ is the Beta function and $Z_n$, $n\ge 2$, are normalisation constants; this extends
to $\beta=-2$ by continuity, i.e. $q^{\rm Aldous}_{-2}(n-1,1)=1$, $n\ge 2$.

For exchangeably labelled Markov branching models $(T_n,n\ge 1)$ it
is convenient to set
\begin{equation}\label{spliteppf} p(n_1,\ldots,n_k):=\frac{m_1!\ldots m_n!}{{n\choose n_1,\ldots,n_k}}q((n_1,\ldots,n_k)^\downarrow),\quad n_j\ge1,j\in[k];k\ge 2:\ n=n_1+\ldots+n_k,
\end{equation}
where $(n_1,\ldots,n_k)^\downarrow$ is the decreasing rearrangement and $m_r$
the number of $r$s of the sequence $(n_1,\ldots,n_k)$.
The function $p$ is called \em exchangeable partition probability function (EPPF) \em and gives the probability that the
branching adjacent to the {\sc root} splits into tree components with label
sets $\{A_1,\ldots,A_k\}$ partitioning $[n]$, with \em block sizes \em $n_j=\#A_j$. Note
that $p$ is invariant under permutations of its arguments. It was shown in
\cite{MPW} that Aldous's beta-splitting models for $\beta>-2$ are the only \em
binary \em Markov branching models for which the EPPF is of Gibbs type
\beqs p^{\rm Aldous}_{-1-\alpha}(n_1,n_2)=\frac{w_{n_1}w_{n_2}}{Z_{n_1+n_2}},\quad n_1\ge 1,n_2\ge 1,\qquad\mbox{in particular }w_n=\frac{\Gamma(n-\alpha)}{\Gamma(1-\alpha)},
\eeqs
and that the \em multifurcating \em Gibbs models are an \em extended \em
Ewens-Pitman two-parameter family of random partitions, $0\le\alpha\le 1$,
$\theta\ge -2\alpha$, or $-\infty\le\alpha<0$, $\theta=-m\alpha$ for some integer $m\ge 2$,
\begin{equation} p_{\alth}^{\rm PD^*}(n_1,\ldots,n_k)=\frac{a_k}{Z_n}\prod_{j=1}^kw_{n_j},\quad
      \mbox{where }w_n=\frac{\Gamma(n-\alpha)}{\Gamma(1-\alpha)}\mbox{ and }a_k=\alpha^{k-2}\frac{\Gamma(k+\theta/\alpha)}{\Gamma(2+\theta/\alpha)},
\label{EPmod}
\end{equation}
boundary cases by continuity. Ford \cite{For-05} introduced a different
\em binary \em model, the \em alpha model\em, using simple
sequential growth rules starting from the unique elements
$T_1\in\bT_1$ and $T_2\in\bT_2$:
\begin{enumerate}\item[(i)$^{\rm F}$] given $T_n$ for $n\ge 2$, assign a weight
  $1-\alpha$ to each of the $n$ edges adjacent to a leaf, and a weight
  $\alpha$ to each of the $n-1$ other edges;
\item[(ii)$^{\rm F}$] select at random with probabilities proportional to the
  weights assigned by step (i)$^{\rm F}$, an edge of $T_n$, say $a_n\rightarrow c_n$
  directed away from the {\sc root};
\item[(iii)$^{\rm F}$] to create $T_{n+1}$ from $T_n$, replace
  $a_n\rightarrow c_n$ by three edges $a_n\rightarrow b_n$,
  $b_n\rightarrow c_n$ and $b_n\rightarrow n+1$ so that two new edges connect
  the two vertices $a_n$ and $c_n$ to a new branch point $b_n$ and a further
  edge connects $b_n$ to a new leaf labelled $n+1$.
\end{enumerate}
It was shown in \cite{For-05} that these trees are Markov branching
trees but that the labelling is not exchangeable. The splitting rule
was calculated and shown to coincide with Aldous's beta-splitting
rules if and only if $\alpha=0$, $\alpha=1/2$ or $\alpha=1$,
interpolating differently between Aldous's corresponding models for
$\beta=0$, $\beta=-3/2$ and $\beta=-2$. This study was taken further
in \cite{HMPW,PW2}.

In this paper, we introduce a new model by extending the simple
sequential growth rules to allow \em multifurcation\em. Specifically, we
also assign weights to \em vertices \em as follows, cf. Figure \ref{fig1}:
\begin{enumerate}\item[(i)] given $T_n$ for $n\ge 2$, assign a weight
  $1-\alpha$ to each of the $n$ edges adjacent to a leaf, a weight
  $\gamma$ to each of the $n-1$ other edges, and a weight $(k-1)\alpha-\gamma$
  to each vertex of degree $k+1\ge 3$;
\item[(ii)] select at random with probabilities proportional to the
  weights assigned by step (i),
  \begin{itemize}\item an edge of $T_n$, say $a_n\rightarrow c_n$
    directed away from the {\sc root},
  \item or, as the case may be, a vertex of $T_n$, say $v_n$;
  \end{itemize}
\item[(iii)] to create $T_{n+1}$ from $T_n$, do the following:
  \begin{itemize}\item if an edge $a_n\rightarrow c_n$ was selected, replace
    it by three edges $a_n\rightarrow b_n$, $b_n\rightarrow c_n$ and
    $b_n\rightarrow n+1$ so that two new edges connect the two vertices $a_n$
    and $c_n$ to a new branch point $b_n$ and a further edge connects $b_n$ to
    a new leaf labelled $n+1$;
  \item if a vertex $v_n$ was selected, add an edge $v_n\rightarrow n+1$ to a
    new leaf labelled $n+1$.\pagebreak[2]
  \end{itemize}
\end{enumerate}
\begin{figure}[t]
\begin{center}
\hspace{-0.0cm}\includegraphics[scale=0.4]{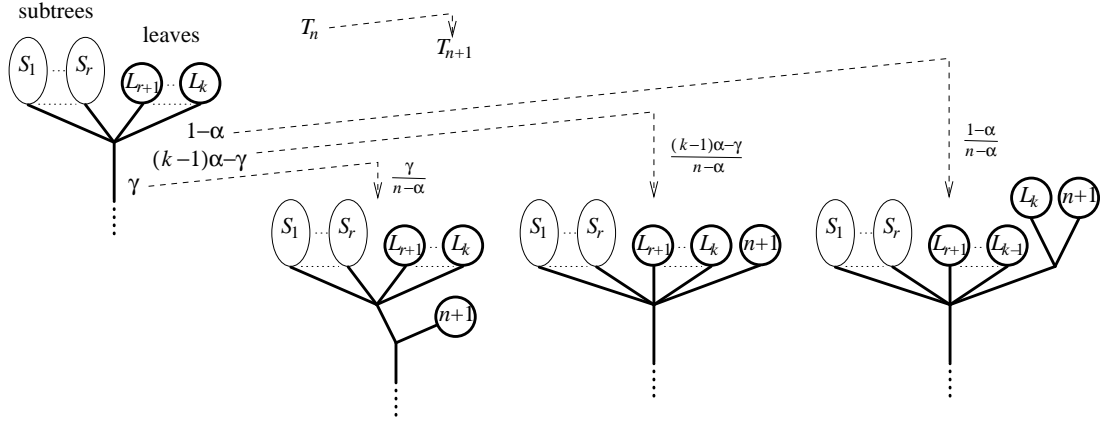}
\end{center}
\vspace{-0.5cm}
\caption{Sequential growth rule: displayed is one branch point of $T_n$ with
degree $k+1$, hence vertex weight $(k-1)\alpha-\gamma$, with $k-r$ leaves
$L_{r+1},\ldots,L_k\in[n]$ and $r$ bigger subtrees $S_1,\ldots,S_r$ attached
to it; all edges also carry weights, weight $1-\alpha$ and $\gamma$ are
displayed here for one leaf edge and one inner edge only; the three associated
possibilities for $T_{n+1}$ are displayed.}
\label{fig1}
\end{figure}
We call the resulting model the \em alpha-gamma model\em. These growth
rules satisfy the rules of probability for all $0\le\alpha\le 1$ and
$0\le\gamma\le\alpha$. They contain the growth rules of the alpha
model for $\gamma=\alpha$. They also contain growth rules for a
model \cite{Mar-08,Mie-03} based on the stable tree of Duquesne and Le
Gall \cite{DuL-02}, for the cases $\gamma=1-\alpha$, $1/2\le\alpha<1$,
where all edges are given the same weight; we show here that these
cases $\gamma=1-\alpha$, $1/2\le\alpha\le 1$, as well as
$\alpha=\gamma=0$ form the intersection with the extended
Ewens-Pitman-type two-parameter family of models (\ref{EPmod}).

\begin{prop}\label{prop1} Let $(T_n,n\ge 1)$ be alpha-gamma trees with
  distributions as implied by the sequential growth rules {\rm (i)-(iii)} for some $0\le\alpha\le 1$ and $0\le\gamma\le\alpha$. Then
  \begin{enumerate}\item[\rm(a)] the delabelled trees $T_n^\circ$, $n\ge 1$, have the Markov branching property. The
    splitting rules are
  \begin{equation}\label{split} q_{\alpha,\gamma}^{\rm seq}(n_1,\ldots,n_k)\quad\propto\quad\left(\gamma+(1-\alpha-\gamma)\frac{1}{n(n-1)}\sum_{i\neq j}n_in_j\right)q_{\alpha,-\alpha-\gamma}^{\rm PD^*}(n_1,\ldots,n_k),
  \end{equation}
  in the case $0\le\alpha<1$, where $q_{\alpha,-\alpha-\gamma}^{\rm PD^*}$ is the splitting rule associated
  via (\ref{spliteppf}) with $p_{\alpha,-\alpha-\gamma}^{\rm PD^*}$, the Ewens-Pitman-type EPPF given in (\ref{EPmod}), and LHS $\propto$ RHS means equality up to a multiplicative constant depending on $n$ and $(\alpha,\gamma)$ that makes the LHS a probability function;
  \item[\rm(b)] the labelling of $T_n$ is exchangeable for all $n\ge 1$ if and only if $\gamma=1-\alpha$, $1/2\le\alpha\le 1$.
  \end{enumerate}
\end{prop}

For any function $(n_1,\ldots,n_k)\mapsto q(n_1,\ldots,n_k)$ that is
a probability function for all fixed $n=n_1+\ldots+n_k$, $n\ge 2$,
we can construct a Markov branching model $(T_n^\circ,n\ge 1)$. A
condition called \em sampling consistency \em \cite{Ald-93} is to
require that the tree $T_{n,-1}^\circ$ constructed from $T_n^\circ$
by removal of a uniformly chosen leaf (and the adjacent branch point
if its degree is reduced to 2) has the same distribution as
$T_{n-1}^\circ$, for all $n\ge 2$. This is appealing for
applications with incomplete observations. It was shown in
\cite{HMPW} that all sampling consistent splitting rules admit an
integral representation $(c,\nu)$ for an erosion coefficient $c\ge
0$ and a dislocation measure $\nu$ on
$\cS^\downarrow=\{s=(s_i)_{i\ge 1}:s_1\ge s_2\ge\ldots\ge
0,s_1+s_2+\ldots\le 1\}$ with $\nu(\{(1,0,0,\ldots)\})=0$ and
$\int_{\cS^\downarrow}(1-s_1)\nu(ds)<\infty$ as in Bertoin's
continuous-time fragmentation theory \cite{Ber-hom,Ber-ss,Ber-book}. In the most
relevant case when $c=0$ and
$\nu(\{s\in\cS^\downarrow:s_1+s_2+\ldots<1\})=0$, this
representation is
\begin{equation}\label{Kingman}
p(n_1,\ldots,n_k)=\frac{1}{\widetilde{Z}_n}\int_{\cS^\downarrow}\sum_{{\QATOP{i_1,
\ldots,i_{k}\ge 1}\atop{\mbox{\scriptsize
  distinct}}}}\prod_{j=1}^ks_{i_j}^{n_j}\nu(ds),\quad n_j\ge1,j\in[k];k\ge 2:\ n=n_1+\ldots+n_k,
\end{equation}
 where $\widetilde{Z}_n=\int_{\cS^\downarrow}(1-\sum_{i\ge 1}s_i^n)\nu(ds)$,
$n\ge 2$, are the normalization constants. The measure $\nu$ is
unique up to a multiplicative constant. In particular, it can be
shown \cite{Mie-03,HPW} that for the Ewens-Pitman EPPFs $p_\alth^{\rm
PD^*}$ we obtain $\nu={\rm PD}^*_\alth(ds)$ of Poisson-Dirichlet
type (hence our superscript ${\rm PD}^*$ for the Ewens-Pitman type EPPF), where for $0<\alpha<1$
and $\theta>-2\alpha$ we can express \beqs
\int_{\cS^\downarrow}f(s){\rm
PD}^*_\alth(ds)=\bE\left(\sigma_1^{-\theta}f\left(\Delta\sigma_{[0,1]}/\sigma_1\right)\right),
\eeqs for an $\alpha$-stable subordinator $\sigma$ with Laplace
exponent $-\log(\bE(e^{-\lambda\sigma_1}))=\lambda^\alpha$ and with
ranked sequence of jumps
$\Delta\sigma_{[0,1]}=(\Delta\sigma_t,t\in[0,1])^\downarrow$. For
$\alpha<1$ and $\theta=-2\alpha$, we have \beqs
\int_{\cS^\downarrow}f(s){\rm
PD}^*_{\alpha,-2\alpha}(ds)=\int_{1/2}^1
f(x,1-x,0,0,\ldots)x^{-\alpha-1}(1-x)^{-\alpha-1}dx. \eeqs Note that
$\nu={\rm PD}^*_{\alpha,\theta}$ is infinite but $\sigma$-finite
with $\int_{\cS^\downarrow}(1-s_1)\nu(ds)<\infty$ for
$-2\alpha\le\theta\le-\alpha$. This is the relevant range for this
paper. For $\theta>-\alpha$, the measure ${\rm PD}^*_\alth$ just
defined is a multiple of the usual Poisson-Dirichlet probability
measure ${\rm PD}_\alth$ on $\cS^\downarrow$, so for the integral
representation of $p_{\alth}^{\rm PD^*}$ we could also take
$\nu={\rm PD}_\alth$ in this case, and this is also an appropriate
choice for the two cases $\alpha=0$ and $m\ge 3$; the case
$\alpha=1$ is degenerate $q_\alth^{\rm PD^*}(1,1,\ldots,1)=1$ (for
all $\theta$) and can be associated with $\nu={\rm
PD}^*_{1,\theta}=\delta_{(0,0,\ldots)}$, see \cite{MPW}.
\begin{theo}\label{thm2} The alpha-gamma-splitting rules
  $q_{\alpha,\gamma}^{\rm seq}$ are sampling consistent. For $0\le\alpha<1$ and
  $0\le\gamma\le\alpha$ the measure $\nu$ in the integral representation can be
  chosen as
  \begin{equation}\label{thm2nu}
   \nu_{\alpha,\gamma}(ds)=\left(\gamma+(1-\alpha-\gamma)\sum_{i\neq j}s_is_j\right){\rm PD}^*_{\alpha,-\alpha-\gamma}(ds).
  \end{equation}
\end{theo}
The case $\alpha=1$ is discussed in Section \ref{sectalpha1}. We refer to Griffiths
\cite{Gri-83} who used discounting of Poisson-Dirichlet measures by
quantities involving $\sum_{i\neq j}s_is_j$ to model genic
selection.

In \cite{HMPW}, Haas and Miermont's self-similar continuum random trees (CRTs) \cite{HM} are
shown to be scaling limits for a wide class of Markov branching
models. See Sections \ref{seccrts} and \ref{sechmpw} for details. This theory applies here to
yield:

\begin{coro}\label{dconv} Let
  $(T_n^\circ,n\ge 1)$ be delabelled alpha-gamma trees, represented as discrete $\bR$-trees with unit edge lengths, for some 
  $0<\alpha<1$ and $0<\gamma\le\alpha$. Then
  \beqs \frac{T_n^\circ}{n^{\gamma}}\rightarrow\cT^{\alpha,\gamma}\qquad\mbox{in distribution for the Gromov-Hausdorff topology,}
  \eeqs
  where the scaling $n^\gamma$ is applied to all edge lengths, and $\cT^{\alpha,\gamma}$ is a $\gamma$-self-similar CRT whose dislocation measure is a multiple of $\nu_{\alpha,\gamma}$.
\end{coro}

We observe that every dislocation measure $\nu$ on $\cS^\downarrow$
gives rise to a measure $\nu^{\rm sb}$ on the space of summable
sequences under which fragment sizes are in a size-biased random
order, just as the ${\rm GEM}_\alth$ distribution can be defined as
the distribution of a ${\rm PD}_\alth$ sequence re-arranged in
size-biased random order \cite{csp}. We similarly define ${\rm
GEM}^*_\alth$ from ${\rm PD}^*_\alth$. One of the advantages of
size-biased versions is that, as for ${\rm GEM}_\alth$, we can
calculate marginal distributions explicitly.

\begin{prop}\label{prop4} For $0<\alpha<1$ and $0\le\gamma<\alpha$,
  distributions $\nu_k^{\rm sb}$ of the first $k\ge 1$ marginals
  of the
  size-biased form $\nu_{\alpha,\gamma}^{\rm sb}$ of $\nu_{\alpha,\gamma}$ are
  given, for $x=(x_1,\ldots,x_k)$, by
  \beqs\hspace{-0.3cm} \nu^{\rm sb}_k(dx)=\left(\gamma+(1-\alpha-\gamma)\left(1-\sum_{i=1}^kx_i^2-\frac{1\!-\!\alpha}{1\!+\!(k\!-\!1)\alpha\!-\!\gamma}\left(1-\sum_{i=1}^kx_i\right)^2\right)\right){\rm GEM}^*_{\alpha,-\alpha-\gamma}(dx).
  \eeqs
\end{prop}
\noindent The other boundary values of parameters are trivial here -- there are at most two non-zero parts.\pagebreak[2]



We can investigate the convergence of Corollary \ref{dconv} when
labels are retained. Since labels are non-exchangeable, in general,
it is not clear how to nicely represent a continuum tree with
infinitely many labels other than by a consistent sequence $\cR_k$
of trees with $k$ leaves labelled $[k]$, $k\ge 1$. See however
\cite{PW2} for developments in the binary case $\gamma=\alpha$ on how
to embed $\cR_k$, $k\ge 1$, in a CRT $\cT^{\alpha,\alpha}$. The
following theorem extends Proposition 18 of \cite{HMPW} to the
multifurcating case.

\begin{theo}\label{LE} Let $(T_n,n\ge 1)$
  be a sequence of trees resulting from the alpha-gamma-tree growth rules for 
  some $0<\alpha<1$ and $0<\gamma\le\alpha$.
  Denote by $R(T_n,[k])$ the subtree of $T_n$ spanned by the {\sc root} and leaves
  $[k]$, reduced by removing degree-2 vertices, represented as discrete
  $\bR$-tree with graph distances in $T_n$ as edge lengths. Then
  \beqs \frac{R(T_n,[k])}{n^{\gamma}}\rightarrow\cR_k\qquad\mbox{a.s. in the
    sense that all edge lengths converge,}
  \eeqs
  for some discrete tree $\cR_k$ with shape $T_k$ and edge lengths specified
  in terms of three random variables, conditionally independent given that $T_k$ has
  $k+\ell$ edges, as $L_kW_k^\gamma D_k$ with
  \begin{itemize}\item $W_k\sim{\rm beta}(k(1-\alpha)+\ell\gamma,(k-1)\alpha-\ell\gamma)$, where ${\rm beta}(a,b)$ is the beta distribution with density $B(a,b)^{-1}x^{a-1}(1-x)^{b-1}1_{(0,1)}(x)$;
    \item $L_k$ with density $\displaystyle\frac{\Gamma(1+k(1-\alpha)+\ell\gamma)}{\Gamma(1+\ell+k(1-\alpha)/\gamma)}s^{\ell+k(1-\alpha)/\gamma}g_\gamma(s)$, where $g_\gamma$ is the Mittag-Leffler density, the density of $\sigma_1^{-\gamma}$ for a subordinator $\sigma$ with Laplace exponent $\lambda^\gamma$;
    \item $D_k\sim{\rm Dirichlet}((1-\alpha)/\gamma,\ldots,(1-\alpha)/\gamma,1,\ldots,1)$, where ${\rm Dirichlet}(a_1,\ldots,a_m)$ is the Dirichlet distribution on $\Delta_m=\{(x_1,\ldots,x_m)\in[0,1]^m:x_1+\ldots+x_m=1\}$ with density of the first $m-1$ marginals proportional to $x_1^{a_1-1}\ldots x_{m-1}^{a_{m-1}-1}(1-x_1-\ldots-x_{m-1})^{a_m-1}$; here $D_k$ contains edge length proportions, first with parameter $(1-\alpha)/\gamma$ for edges adjacent to leaves and then with parameter $1$ for the other edges, each enumerated
      e.g. by depth first search.
  \end{itemize}
\end{theo}

In fact, $1-W_k$ captures the total limiting leaf proportions of
subtrees that are attached on the vertices of $T_k$, and we can
study further how this is distributed between the branch points, see
Section \ref{secbw}.

We conclude this introduction by giving an alternative description of the alpha-gamma
model obtained by adding colouring rules to the alpha model growth rules
(i)$^{\rm F}$-(iii)$^{\rm F}$, so that in $T_n^{\rm col}$ each edge except
those adjacent to leaves has either a blue or a red colour mark.
\begin{enumerate}\item[(iv)$^{\rm col}$] To turn $T_{n+1}$ into a colour-marked tree $T_{n+1}^{\rm col}$, keep the colours of $T_n^{\rm col}$ and do the following: \begin{itemize} \item if an edge $a_n\rightarrow c_n$ adjacent to a leaf was selected, mark $a_n\rightarrow b_n$ blue;
   \item if a red edge $a_n\rightarrow c_n$ was selected, mark both $a_n\rightarrow b_n$ and $b_n\rightarrow c_n$ red;
  \item if a blue edge $a_n\rightarrow c_n$ was selected, mark
    $a_n\rightarrow b_n$ blue; mark $b_n\rightarrow c_n$ red with probability
    $c$ and blue with probability $1-c$;
  \end{itemize}
\end{enumerate}
When $(T_n^{\rm col},n\ge 1)$ has been grown according to (i)$^{\rm F}$-(iii)$^{\rm F}$ and (iv)$^{\rm col}$, crush all red edges, i.e.
\begin{enumerate}\item[(cr)] identify all vertices connected via red edges,
remove all red edges and remove the remaining colour marks; denote the resulting sequence of trees by $(\widetilde{T}_{n},n\ge 1)$;
\end{enumerate}

\begin{prop}\label{prop6} Let $(\widetilde{T}_n,n\ge 1)$ be a sequence of trees according to growth
  rules {\rm (i)}$^{\rm F}$-{\rm(iii)}$^{\rm F}$,{\rm(iv)}$^{\rm col}$ and crushing rule {\rm (cr)}. Then $(\widetilde{T}_n,n\ge 1)$ is a sequence of
  alpha-gamma trees with $\gamma=\alpha(1-c)$.
\end{prop}

The structure of this paper is as follows. In Section 2 we study the discrete trees grown according to the growth rules
(i)-(iii) and establish Proposition \ref{prop6} and Proposition \ref{prop1} as well as the sampling consistency claimed in Theorem \ref{thm2}. Section 3 is devoted to the limiting CRTs, we obtain the dislocation measure stated in Theorem \ref{thm2} and deduce Corollary \ref{dconv}
and Proposition \ref{prop4}. In Section 4 we study the convergence of labelled trees and prove Theorem \ref{LE}.

\section{Sampling consistent splitting rules for the alpha-gamma trees}
\subsection{Notation and terminology of partitions and discrete fragmentation trees}\label{trees}
For $B\subseteq\mathbb{N}$, let $\mathcal{P}_B$ be the \em set of
partitions of $B$ \em into disjoint non-empty subsets called \em blocks\em.
Consider a probability space $(\Omega,\mathcal{F},\mathbb{P})$,
which supports a $\mathcal{P}_B$-valued random partition $\Pi_B$. If
the probability function of $\Pi_B$ only
depends on its block sizes, we call it \em exchangeable\em. Then 
$$ \mathbb{P}(\Pi_B=\{A_1,\ldots ,A_k\})=p(\#A_1,\ldots,\#A_k)\qquad\mbox{for each
partition $\pi=\{A_1,\ldots ,A_k\}\in\mathcal{P}_B$,}$$ 
where $\#A_j$ denotes the block size, i.e. the number of elements of $A_j$. This function
$p$ is called  the \em exchangeable partition probability function \em
(EPPF) of $\Pi_B$. Alternatively, a random partition $\Pi_B$ is
exchangeable if its distribution is invariant under the natural
action on partitions of $B$ by the symmetric group of permutations
of $B$. 

Let $B\subseteq\mathbb{N}$, we say that a partition $\pi\in
\mathcal{P}_B$ is \em finer than \em $\pi'\in \mathcal{P}_B$, and write
$\pi\preceq\pi'$, if any block of $\pi$ is included in some block of
$\pi'$. This defines a partial order $\preceq$ on $\mathcal{P}_B$. A
process or a sequence with values in $\mathcal{P}_B$ is called
refining if it is decreasing for this partial order.
Refining partition-valued processes are naturally related to trees.
Suppose that $B$ is a finite subset of $\mathbb{N}$ and
$\mathbf{t}$ is a collection of subsets of $B$ with an additional
member called the {\sc root} such that
\begin{itemize}
   \item we have $B\in\mathbf{t}$; we call $B$ the \em common ancestor \em of
   $\mathbf{t}$;
   \item we have $\{i\}\in \mathbf{t}$ for all $i\in B$; we call $\{i\}$ 
   a \em leaf \em of $\mathbf{t}$;
   \item for all $A\in\mathbf{t}$ and $C\in \mathbf{t}$, we have either $A\cap C=\varnothing$,
   or $A\subseteq C$ or $C\subseteq A$.
\end{itemize}
If $A\subset C$, then $A$ is called a \em descendant \em of $C$, or $C$ 
an \em ancestor \em of $A$. If for all $D\in \mathbf{t}$ with
$A\subseteq D\subseteq C$ either $A=D$ or $D=C$, we call $A$ a \em child \em
of $C$, or $C$ the \em parent \em of $A$ and denote $C\rightarrow A$. If we equip 
$\mathbf{t}$ with the parent-child relation and also $\mbox{\sc root}\rightarrow B$, then
$\mathbf{t}$ is a rooted connected acyclic graph, i.e. a combinatorial tree. We
denote the space of such trees $\mathbf{t}$ by $\mathbb{T}_B$ and
also $\bT_n=\bT_{[n]}$. For $\ft\in\bT_B$ and $A\in\ft$, the rooted 
subtree $\fs_A$ of $\ft$ with common ancestor $A$ is given by $\fs_A=\{\mbox{\sc root}\}\cup\{C\in\ft:C\subseteq A\}\in\bT_A.$
In particular, we consider the \em subtrees $\fs_j=\fs_{A_j}$ of the common 
ancestor $B$ of $\ft$\em, i.e. the subtrees whose common ancestors $A_j$, $j\in[k]$, are the children of $B$. In other words, $\fs_1,\ldots,\fs_k$ are the rooted connected components of $\ft\setminus\{B\}$.

Let $(\pi(t), t\geq0)$ be a $\mathcal{P}_B$-valued refining process
for some finite $B\subset\mathbb{N}$ with
$\pi(0)=\mathbf{1}_B$ and $\pi(t)=\mathbf{0}_B$ for some
$t>0$, where $\mathbf{1}_B$ is the trivial partition into a single
block $B$ and $\mathbf{0}_B$ is the partition of $B$ into
singletons. We define $\mathbf{t}_\pi=\{\text{\sc
root}\}\cup\{A\subset B: A\in \pi(t)\mbox{ for some $t\geq 0$}\}$ as the
associated \textit{labelled fragmentation tree}.

\begin{defi}\label{lab}\rm Let $B\subset\bN$ with $\#B=n$ and $\ft\in\bT_B$. We associate the 
  relabelled tree 
	$$\ft^\sigma=\{\mbox{\sc root}\}\cup\{\sigma(A):A\in\ft\}\in\bT_n,$$
  for any bijection $\sigma:B\rightarrow[n]$, and the combinatorial tree shape of $\ft$ as the equivalence class 
    $$\ft^\circ=\{\ft^\sigma|\sigma:B\rightarrow[n]\mbox{ bijection}\}\subset\bT_n.$$
  We denote by $\bT_n^\circ=\{\ft^\circ:\ft\in\bT_n\}=\{\ft^\circ:\ft\in\bT_B\}$ the collection of all tree shapes with $n$ leaves, 
  which we will also refer to in their own right as \em unlabelled fragmentation trees\em.
\end{defi}

Note that the number of subtrees of the common ancestor of $\ft\in\bT_n$ and the numbers of leaves in these subtrees are invariants of the equivalence class $\ft^\circ\subset\bT_n$. If $\mathbf{t}^\circ\in\bT_n^\circ$ has subtrees $\mathbf{s}_1^\circ,\ldots,\mathbf{s}_k^\circ$ with $n_1\geq\ldots\geq n_k\geq 1$ leaves, we say that $\mathbf{t}^\circ$ is formed by \em joining
together \em $\mathbf{s}_1^\circ,\ldots,\mathbf{s}_k^\circ$, denoted by
$\mathbf{t}^\circ=\mathbf{s}_1^\circ*\ldots*\mathbf{s}_k^\circ$. We
call the \em composition \em $(n_1,\ldots,n_k)$ of $n$ the \em first split \em of
$\mathbf{t}_n^\circ$. 

With this notation and terminology, a sequence of random trees $T_n^\circ\in\bT_n^\circ$, $n\ge 1$, has the \em Markov branching
property \em if, for all $n\ge 2$,
the tree $T_n^\circ$ has the same distribution as $S_1^\circ*\ldots*S_{K_n}^\circ$, where $N_1\ge\ldots\ge N_{K_n}\ge 1$ form a
random composition of $n$ with $K_n\ge 2$ parts, and conditionally given $K_n=k$ and $N_j=n_j$, the trees $S_j^\circ$, $j\in[k]$, are
independent and distributed as $T_{n_j}^\circ$, $j\in[k]$.  

\subsection{Colour-marked trees and the proof of Proposition \ref{prop6}}
The growth rules (i)$^{\rm F}$-(iii)$^{\rm F}$ construct binary combinatorial trees $T_n^{\rm bin}$ with vertex set 
\beqs V=\{\mbox{\sc root}\}\cup[n]\cup\{b_1,\ldots,b_{n-1}\}
\eeqs
and an edge set $E\subset V\times V$. We write $v\rightarrow w$ if $(v,w)\in E$. In Section \ref{trees}, we identify leaf $i$ with 
the set $\{i\}$ and vertex $b_i$ with $\{j\in[n]:b_i\rightarrow\ldots\rightarrow j\}$, the edge set $E$ then being identified by the parent-child relation. In this framework, a \em colour mark \em for an edge 
$v\rightarrow b_i$ can be assigned to the vertex $b_i$, so that a \em coloured binary tree \em as constructed in 
(iv)$^{\rm col}$ can be represented by
\beqs V^{\rm col}=\{\mbox{\sc root}\}\cup[n]\cup\{(b_1,\chi_n(b_1)),\ldots,(b_{n-1},\chi_n(b_{n-1}))\}
\eeqs
for some $\chi_n(b_i)\in\{0,1\}$, $i\in[n-1]$, where $0$ represents red and $1$ represents blue.
 


\begin{proof}[Proof of Proposition \ref{prop6}]
We only need to check that the growth rules (i)$^{\rm F}$-(iii)$^{\rm F}$ and (iv)$^{\rm col}$ for $(T_n^{\rm col},n\ge 1)$ imply
that the uncoloured multifurcating trees $(\widetilde{T}_n,n\ge 1)$ obtained from $(T_n^{\rm col},n\ge 1)$ via crushing (cr)
satisfy the growth rules (i)-(iii). Let therefore $\ft^{\rm col}_{n+1}$ be a tree with $\bP(T_{n+1}^{\rm col}=\ft^{\rm col}_{n+1})>0$. It is easily seen that there is a unique tree $\ft^{\rm col}_n$, a unique insertion edge $a_n^{\rm col}\rightarrow c_n^{\rm col}$ in $\ft^{\rm col}_n$ and, if any, a unique colour $\chi_{n+1}(c_n^{\rm col})$ to create $\ft^{\rm col}_{n+1}$ from $\ft^{\rm col}_n$. Denote the trees obtained from 
$\ft^{\rm col}_n$ and $\ft^{\rm col}_{n+1}$ via crushing (cr) by $\ft_n$ and $\ft_{n+1}$. If $\chi_{n+1}(c_n^{\rm col})=0$, 
denote by $k+1\ge 3$ the degree of the branch point of $\ft_n$ with which $c_n^{\rm col}$ is identified in the first step of the crushing (cr). 
\begin{itemize}\item If the insertion edge is a leaf edge ($c_n^{\rm col}=i$ for some $i\in[n]$), we obtain 
  $$\bP(\widetilde{T}_{n+1}=\ft_{n+1}|\widetilde{T}_n=\ft_n,T_n^{\rm col}=\ft_n^{\rm col})=(1-\alpha)/(n-\alpha).$$
  \item If the insertion edge has colour blue ($\chi_n(c_n^{\rm col})=1$) and also $\chi_{n+1}(c_n^{\rm col})=1$, we obtain
  $$\bP(\widetilde{T}_{n+1}=\ft_{n+1}|\widetilde{T}_n=\ft_n,T_n^{\rm col}=\ft_n^{\rm col})=\alpha(1-c)/(n-\alpha).$$
  \item If the insertion edge has colour blue ($\chi_n(c_n^{\rm col})=1$), but $\chi_{n+1}(c_n^{\rm col})=0$, or if the insertion
  edge has colour red ($\chi_n(c_n^{\rm col})=0$, and then necessarily $\chi_{n+1}(c_n^{\rm col})=0$ also), we obtain 
  $$\bP(\widetilde{T}_{n+1}=\ft_{n+1}|\widetilde{T}_n=\ft_n,T_n^{\rm col}=\ft_n^{\rm col})=(c\alpha+(k-2)\alpha)/(n-\alpha),$$ 
  because apart from $a_n^{\rm col}\rightarrow c_n^{\rm col}$, there are $k-2$ other edges in $\ft^{\rm col}_n$, where insertion
  and crushing also create $\ft_{n+1}$. 
\end{itemize} 
Because these conditional probabilities do not depend on $\ft_n^{\rm col}$ and have the form required, we conclude that $(\widetilde{T}_n,n\ge 1)$ obeys the 
growth rules (i)-(iii) with $\gamma=\alpha(1-c)$.
\end{proof}

\subsection{The Chinese Restaurant Process}
An important tool in this paper is the Chinese Restaurant Process (CRP), a 
partition-valued process $(\Pi_n,n\ge 1)$ due to Dubins and Pitman, see \cite{csp}, which
generates the Ewens-Pitman two-parameter family of exchangeable 
random partitions $\Pi_\infty$ of $\mathbb{N}$. In the restaurant framework,
each block of a partition is represented by a \em table \em and each 
element of a block by a \em customer \em at a table. The construction rules 
are the
following. The first customer sits at the first table and the
following customers will be seated at an occupied table or a new one.
Given $n$ customers at $k$ tables with $n_j\ge 1$
customers at the $j$th table, customer $n+1$ will be placed at
the $j$th table with probability $(n_j-\alpha)/(n+\theta)$, and
at a new table with probability
$(\theta+k\alpha)/(n+\theta)$. The parameters $\alpha$ and $\theta$ 
can be chosen as either $\alpha<0$ and $\theta=-m\alpha$ for
some $m\in\bN$ or $0\leq \alpha\leq 1$ and $\theta>-\alpha$. We
refer to this process as the CRP with 
$(\alpha,\theta)$-\em seating plan\em.

In the CRP $(\Pi_n,n\ge 1)$ with $\Pi_n\in\cP_{[n]}$, 
we can study the block sizes, which leads us to consider the 
proportion of each table relative to the total number of customers. These proportions 
converge to \textit{limiting frequencies} as follows. 

\begin{lemm}[Theorem 3.2 in \cite{csp}]\label{crp1}
For each pair of parameters $(\alpha,\theta)$ subject to the
constraints above, the Chinese restaurant with the $(\alpha,\theta)$-seating plan generates an exchangeable random partition $\Pi_\infty$
of $\mathbb{N}$. The corresponding EPPF is
$$p_{\alpha,\theta}^{\rm PD}(n_1,\ldots ,n_k)=\frac{\alpha^{k-1}\Gamma(k+\theta/\alpha)\Gamma(1+\theta)}
{\Gamma(1+\theta/\alpha)\Gamma(n+\theta)}\prod_{i=1}^k\frac{\Gamma(n_i-\alpha)}{\Gamma(1-\alpha)},\quad n_i\ge 1,i\in[k];k\ge 1:\ \mbox{$\sum n_i=n$,}$$
boundary cases by continuity. The corresponding limiting frequencies of block sizes, in size-biased order
of least elements, are ${\rm GEM}_{\alpha,\theta}$ and can be represented as
$$(\tilde{P_1},\tilde{P_2},\ldots )=(W_1,\overline{W}_1W_2,\overline{W}_1\overline{W}_2W_3,\ldots )$$
  where the $W_i$ are independent, $W_i$ has ${\rm beta}(1-\alpha,
  \theta+i\alpha)$ distribution, and $\overline{W}_i:=1-W_i.$ The distribution 
  of the associated ranked sequence of limiting frequencies is Poisson-Dirichlet ${\rm PD}_{\alpha,\theta}$.

\end{lemm}
We also associate with the EPPF $p_{\alpha,\theta}^{\rm PD}$ the distribution $q_{\alpha,\theta}^{\rm PD}$ of block sizes in
decreasing order via (\ref{spliteppf}) and, because the Chinese restaurant EPPF is \em not \em the EPPF of a splitting rule 
leading to $k\ge 2$ block (we use notation $q_{\alpha,\theta}^{\rm PD^*}$ for the splitting rules induced by conditioning on $k\ge 2$ blocks), but can lead to a single block, we also set $q_{\alpha,\theta}^{\rm PD}(n)=p_{\alpha,\theta}^{\rm PD}(n)$.

The asymptotic properties of the number $K_n$ of blocks of $\Pi_n$ 
under the $(\alpha,\theta)$-seating plan depend on $\alpha$: if $\alpha<0$ and $\theta=-m\alpha$ for some $m\in\bN$, then $K_n=m$ for all sufficiently large $n$ a.s.; if $\alpha=0$ and $\theta>0$, then $\lim_{n\rightarrow\infty}K_n/\log n=\theta$ a.s.
The most relevant case for us is $\alpha>0$.

\begin{lemm}[Theorem 3.8 in \cite{csp}]\label{crp2}
For $0<\alpha<1$, $\theta>-\alpha$, ,
$$\frac{K_n}{n^\alpha}\rightarrow S\qquad\mbox{a.s. as $n\rightarrow\infty$,}$$ 
where $S$ has a continuous density on $(0,\infty)$ given by
$$ \frac{d}{ds}\mathbb{P}(S\in ds)=\frac{\Gamma(\theta+1)}{\Gamma(\theta/\alpha+1)}s^{-\theta/\alpha}g_\alpha(s),$$ 
and $g_\alpha$ is the density of the Mittag-Leffler distribution with $p$th moment $\Gamma(p+1)/\Gamma(p\alpha+1)$.
\end{lemm}

As an extension of the CRP, Pitman and Winkel in
\cite{PW2} introduced the \em ordered \em CRP. Its
seating plan is as follows. The tables are ordered from left to
right. Put the second table to the right of the first with
probability $\theta/(\alpha+\theta)$ and to the left with
probability $\alpha/(\alpha+\theta)$. 
Given $k$ tables, put the $(k+1)$st table to the right of the
right-most table with probability $\theta/(k\alpha+\theta)$ and to the left of the left-most or between two adjacent tables  
with probability $\alpha/(k\alpha+\theta)$ each.

A composition of $n$ is a sequence $(n_1,\ldots,n_k)$ of positive
numbers with sum $n$. A sequence of random compositions $\cC_n$ of
$n$ is called \textit{regenerative} if conditionally given that the
first part of $\cC_n$ is $n_1$, the remaining parts of
$\cC_n$ form a composition of $n-n_1$ with the same distribution
as $\cC_{n-n_1}$. Given any decrement matrix $(q^{\rm dec}(n,m), 1\leq m\leq
n)$, there is an associated sequence $\cC_n$ of regenerative random
compositions of $n$ defined by specifying that $q^{\rm dec}(n,\cdot)$ is the
distribution of the first part of $\cC_n$. Thus for each composition
$(n_1,\ldots,n_k)$ of $n$,
$$\mathbb{P}(\cC_n=(n_1,\ldots,n_k))=q^{\rm dec}(n,n_1)q^{\rm dec}(n-n_1,n_2)\ldots q^{\rm dec}(n_{k-1}+n_k,n_{k-1})q^{\rm dec}(n_k,n_k).$$

\begin{lemm}[Proposition 6 (i) in \cite{PW2}]\label{OCRL}
For each $(\alpha,\theta)$ with $0<\alpha<1$ and $\theta\geq 0$,
denote by $\cC_n$ the composition of block sizes in the ordered 
Chinese restaurant partition with parameters
$(\alpha,\theta)$. Then $(\cC_n, n\geq 1)$ is regenerative, with
decrement matix
\begin{equation}
  q_{\alpha,\theta}^{\rm dec}(n,m)={n \choose
  m}\frac{(n-m)\alpha+m\theta}{n}\frac{\Gamma(m-\alpha)\Gamma(n-m+\theta)}{\Gamma(1-\alpha)\Gamma(n+\theta)}\ \ \
  (1\leq m\leq n).
\end{equation}
\end{lemm}

\subsection{The splitting rule of alpha-gamma trees and the proof of Proposition \ref{prop1}}

Proposition \ref{prop1} claims that the unlabelled alpha-gamma trees $(T_n^\circ,n\ge 1)$ have the Markov 
branching property, identifies the splitting rule and studies the exchangeability of labels. In 
preparation of the proof of the Markov branching property, we use CRPs to compute the probability function of the first split of $T_n^\circ$ in Proposition \ref{prop10}. We will then establish 
the Markov branching property from a spinal decomposition result (Lemma \ref{spinaldec}) for
$T_n^\circ$.

\begin{prop}\label{prop10} Let $T_n^\circ$ be an unlabelled alpha-gamma tree for some $0\le\alpha<1$ and $0\le\gamma\le\alpha$, then the probability function of the first split of $T_n^\circ$ is
  \beqs q^{\rm seq}_{\alpha,\gamma}(n_1,\ldots,n_k)=\frac{Z_n\Gamma(1-\alpha)}{\Gamma(n-\alpha)}\left(\gamma+(1-\alpha-\gamma)\frac{1}{n(n-1)}\sum_{i\neq j} n_i
n_j\right
)q_{\alpha,-\alpha-\gamma}^{\rm
PD^*}(n_1,\ldots ,n_k),
  \eeqs
  $n_1\ge\ldots\ge n_k\ge 1$, $k\ge 2$: $n_1+\ldots+n_k=n$, where $Z_n$ is the normalisation constant in (\ref{EPmod}). 
\end{prop}
\begin{proof}
We start from the growth rules of the labelled alpha-gamma trees $T_n$.
Consider the \em spine \em $\mbox{\sc root}\rightarrow v_1\rightarrow\ldots\rightarrow v_{L_{n-1}}\rightarrow 1$ of $T_n$,
and the \em spinal subtrees \em $S_{ij}^{\rm sp}$, $1\le i\le L_{n-1}$, $1\le j\le K_{n,i}$, not containing 1 of the spinal vertices $v_i$, 
$i\in[L_{n-1}]$. By joining together the subtrees of the spinal vertex $v_i$ we form the $i$th \em spinal bush \em 
$S_i^{\rm sp}=S_{i1}^{\rm sp}*\ldots*S_{iK_{n,i}}^{\rm sp}$. Suppose a bush $S_i^{\rm sp}$ consists of $k$
subtrees with $m$ leaves in total, then its weight will be
$m-k\alpha-\gamma+k\alpha=m-\gamma$ according to growth
rule (i) -- recall that the total weight of the tree $T_n$ is $n-\alpha$.

Now we consider each bush as a table, each leaf $n=2,3,\ldots$ as a customer, 2 being the first customer. Adding 
a new leaf to a bush or to an edge on the spine corresponds
to adding a new customer to an existing or to a new table. The weights
are such that we construct an ordered Chinese restaurant partition of $\bN\setminus\{1\}$ with 
parameters $(\gamma, 1-\alpha)$. 

Suppose that the first split of $T_n$ is into tree components with
numbers of leaves $n_1\ge\ldots\ge n_k\ge 1$. Now suppose further that leaf 1 
is in the subtree with $n_i$ leaves in the first split, then the first 
spinal bush $S_1^{\rm sp}$ will have $n-n_i$ leaves. Notice that this event is 
equivalent to that of $n-n_i$ customers sitting at the first table 
with a total of $n-1$ customers present, in the terminology of the ordered 
CRP. According to Lemma \ref{OCRL}, the probability of this is
\begin{eqnarray}
q^{\rm dec}_{\gamma, 1-\alpha}(n-1,n-n_i)&=&{n-1 \choose
n-n_i}\frac{(n_i-1)\gamma+(n-n_i)(1-\alpha)}{n-1}\frac{\Gamma(n_i-\alpha)\Gamma(n-n_i-\gamma)}{\Gamma(n-\alpha)\Gamma(1-\gamma)}\nonumber\\
&=&{n\choose n-n_i}\left
(\frac{n_i}{n}\gamma+\frac{n_i(n-n_i)}{n(n-1)}(1-\alpha-\gamma)\right)\frac{\Gamma(n_i-\alpha)\Gamma(n-n_i-\gamma)}{\Gamma(n-\alpha)\Gamma(1-\gamma)}.\nonumber
\end{eqnarray}
Next consider the probability that the first bush $S_1^{\rm sp}$
joins together subtrees with $n_1\ge \ldots\ge n_{i-1}\ge n_{i+1}\ge \ldots n_k\ge 1$ leaves
conditional on the event that leaf 1 is in a subtree with $n_i$ leaves. The first bush has a
weight of $n-n_i-\gamma$ and each subtree in it has a weight of
$n_j-\alpha, j\neq i$. Consider these $k-1$ subtrees as tables and
the leaves in the first bush as customers. According to the growth
procedure, they form a second (unordered, this time) Chinese restaurant partition with
parameters $(\alpha, -\gamma)$, whose EPPF is
\begin{equation}
p_{\alpha,-\gamma}^{\rm PD}(n_1,\ldots,n_{i-1},n_{i+1},\ldots,n_k)=\frac{\alpha^{k-2}\Gamma(k-1-\gamma/\alpha)\Gamma(1-\gamma)}{\Gamma(1-\gamma/\alpha)\Gamma(n-n_i-\gamma)}\prod_{j\in[k]\setminus\{i\}}\frac{\Gamma(n_j-\alpha)}{\Gamma(1-\alpha)}.\nonumber
\end{equation}
Let $m_j$ be the number of $j$s in the sequence of $(n_1,\ldots,n_k)$.
Based on the exchangeability of the second Chinese restaurant partition, the
probability that the first bush
 consists of subtrees with $n_1\ge\ldots\ge n_{i-1}\ge n_{i+1}\ge\ldots\ge n_k\ge 1$ leaves
 conditional on the event that leaf 1 is in one of the $m_{n_i}$ subtrees with $n_i$ leaves
will be
\begin{equation}
\frac{m_{n_i}}{m_1!\ldots m_n!}{n-n_i\choose
n_1,\ldots,n_{i-1},n_{i+1},\ldots,n_k}p_{\alpha,-\gamma}^{\rm PD}(n_1,\ldots,n_{i-1},n_{i+1},\ldots,n_k).\nonumber
\end{equation}
Thus the joint
probability that the first split is $(n_1,\ldots,n_k)$ and that leaf 1
is in a subtree with $n_i$ leaves is,
\begin{eqnarray}\label{pni}
&&\hspace{-1cm}\frac{m_{n_i}}{m_1!\ldots m_n!}{n-n_i\choose
n_1,\ldots,n_{i-1},n_{i+1},\ldots,n_k}q^{\rm dec}_{\gamma,
1-\alpha}(n-1,n-n_i)p^{\rm PD}_{\alpha,-\gamma}(n_1,\ldots,n_{i-1},n_{i+1},\ldots,n_k)\nonumber\\
&=& m_{n_i}\left
(\frac{n_i}{n}\gamma+\frac{n_i(n-n_i)}{n(n-1)}(1-\alpha-\gamma)\right)
\frac{Z_n\Gamma(1-\alpha)}{\Gamma(n-\alpha)}q_{\alpha,-\alpha-\gamma}^{\rm
PD^*}(n_1,\ldots,n_k).
\end{eqnarray}

 Hence the
splitting rule will be the sum of (\ref{pni}) for all \em different \em
$n_i$ (not $i$) in $(n_1,\ldots ,n_k)$, but they contain factors $m_{n_i}$, so we can write it as sum over $i\in[k]$:
\begin{eqnarray*}
q_{\alpha,\gamma}^{\rm seq}(n_1,\ldots ,n_k) 
&=&\left(\sum_{i=1}^k \left
(\frac{n_i}{n}\gamma+\frac{n_i(n-n_i)}{n(n-1)}(1-\alpha-\gamma)\right)\right )\frac{Z_n\Gamma(1-\alpha)}{\Gamma(n-\alpha)}q_{\alpha,-\alpha-\gamma}^{\rm PD^*}(n_1,\ldots ,n_k)\\
&=&\left(\gamma+(1-\alpha-\gamma)\frac{1}{n(n-1)}\sum_{i\neq j} n_i
n_j\right
)\frac{Z_n\Gamma(1-\alpha)}{\Gamma(n-\alpha)}q_{\alpha,-\alpha-\gamma}^{\rm
PD^*}(n_1,\ldots ,n_k).
\end{eqnarray*}\vspace{-0.6cm}

\end{proof}

We can use the nested Chinese restaurants described in the proof to study the subtrees of the spine of $T_n$. We have decomposed
$T_n$ into the subtrees $S_{ij}^{\rm sp}$ of the spine from the {\sc root} to 1 and can, conversely, build $T_n$ from 
$S_{ij}^{\rm sp}$, for which we now introduce notation
\beqs T_n=\coprod_{i,j}S_{ij}^{\rm sp}.
\eeqs 
We will also write $\coprod_{i,j}S_{ij}^\circ$ when we join together unlabelled trees $S_{ij}^\circ$ along a spine. The following unlabelled version of a spinal decomposition theorem will entail the Markov branching property.

\begin{lemm}[Spinal decomposition]\label{spinaldec} Let $(T_n^{\circ1},n\ge 1)$ be alpha-gamma trees, delabelled apart from label
  1. For all $n\ge 2$, the tree $T_n^{\circ1}$ has the same distribution as $\coprod_{i,j}S_{ij}^\circ$, where 
  \begin{itemize}\item $\cC_{n-1}=(N_1,\ldots,N_{L_{n-1}})$ is a regenerative composition with decrement matrix 
      $q_{\gamma,1-\alpha}^{\rm dec}$,
    \item conditionally given $L_{n-1}=\ell$ and $N_i=n_i$, $i\in[\ell]$, the sizes $N_{i1}\ge\ldots\ge N_{iK_{n,i}}\ge 1$ form 
	  random compositions of $n_i$ with distribution $q_{\alpha,-\gamma}^{PD}$, independently for $i\in[\ell]$,  
    \item conditionally given also $K_{n,i}=k_i$ and $N_{ij}=n_{ij}$, the trees $S_{ij}^\circ$, $j\in[k_i]$, $i\in[\ell]$, are 
      independent and distributed as $T_{n_{ij}}^\circ$.
  \end{itemize}
\end{lemm}
\begin{proof} For an induction on $n$, note that the claim is true for $n=2$, since $T_n^{\circ1}$ and $\coprod_{i,j}S_{ij}^\circ$
  are deterministic for $n=2$. Suppose then that the claim is true for some $n\ge 2$ and consider $T_{n+1}^\circ$. 

  The growth rules (i)-(iii) of the labelled alpha-gamma tree $T_n$ are such that 
  \begin{itemize}\item leaf $n+1$ is inserted into a new bush or any of the bushes $S_i^{\rm sp}$ selected according to the rules
      of the ordered CRP with $(\gamma,1-\alpha)$-seating plan, 
    \item further into a new subtree or any of the subtrees $S_{ij}^{\rm sp}$ of the selected bush $S_i^{\rm sp}$ according to the
      rules of a CRP with $(\alpha,-\gamma)$-seating plan,
    \item and further within the subtree $S_{ij}^{\rm sp}$ according to the weights assigned by (i) and growth rules (ii)-(iii).
  \end{itemize}
  These selections do not depend on $T_n$ except via $T_n^{\circ1}$. In fact, since labels do not feature in the growth rules
  (i)-(iii), they are easily seen to induce growth rules for partially labelled alpha-gamma trees $T_n^{\circ1}$, and also for
  unlabelled alpha-gamma trees such as $S_{ij}^\circ$. 
  
  From these observations and the induction hypothesis, we deduce the claim for $T_{n+1}^\circ$. 
\end{proof}

\begin{proof}[Proof of Proposition \ref{prop1}] 
  ${\rm(a)}$  Firstly, the distributions of the first splits of the unlabelled alpha-gamma trees $T_n^\circ$ were calculated
  in Proposition \ref{prop10}, for $0\le\alpha<1$ and $0\le\gamma\le\alpha$.  

  Secondly, let $0\le\alpha\le 1$ and $0\le\gamma\le\alpha$. By the regenerative property of the spinal composition $\cC_{n-1}$ and the conditional
  distribution of $T_n^{\circ1}$ given $\cC_{n-1}$ identified in Lemma \ref{spinaldec}, we obtain that given $N_1=m$, 
  $K_{n,1}=k_1$ and $N_{1j}=n_{1j}$, $j\in[k_1]$, the subtrees $S_{1j}^\circ$, $j\in[k_1]$, are independent alpha-gamma trees
  distributed as $T_{n_{1j}}^\circ$, also independent of the remaining tree $S_{1,0}:=\coprod_{i\ge 2,j}S_{ij}^\circ$, which, by
  Lemma \ref{spinaldec}, has the same distribution as $T_{n-m}^\circ$. 
  
  This is equivalent to saying that conditionally given that the first split is into subtrees with $n_1\ge \ldots\ge n_{i}\ge \ldots\ge n_k\ge 1$ leaves and that leaf 1 is
  in a subtree with $n_i$ leaves, the delabelled subtrees $S_1^\circ,\ldots,S_k^\circ$ of the common ancestor are independent and
  distributed as $T_{n_j}^{\circ}$ respectively, $j\in[k]$. Since this conditional distribution does not depend on $i$, we have
  established the Markov branching property of $T_n^\circ$. 


  (b) Notice that if $\gamma=1-\alpha$, the alpha-gamma model is the model related to stable trees, the labelling of
  which is known to be exchangeable, see Section \ref{sectstable}.

  On the other hand, if $\gamma\neq 1-\alpha$, let us turn to look at
  the distribution of $T_3$.

\setlength{\unitlength}{0.5cm}
\begin{picture}(30,6)
\multiput(10,1)(10,0){2}{\line(0,1){2.8}} \multiput(10,
2.4)(10,0){2}{\line(1,1){1}} \multiput(10,
3.8)(10,0){2}{\line(1,1){1}} \multiput(10,
3.8)(10,0){2}{\line(-1,1){1}} \put(8.8,5){1} \put(10.8,5){2}
\put(10.8,3.6){3} \put(18.8,5){1} \put(20.8,5){3} \put(20.8,3.6){2}
\put(7.5,0){Probability:\ $\frac{\gamma}{2-\alpha}$}
\put(17.5,0){Probability:\ $\frac{1-\alpha}{2-\alpha}$}
\end{picture}

We can see the probabilities of the two labelled tree in the above
picture are different although they have the same unlabelled tree.
So if $\gamma\neq 1-\alpha$, $T_n$ is not exchangeable.
\end{proof}

\subsection{Sampling consistency and strong sampling consistency}\label{sectsc}
Recall that an unlabelled Markov branching tree $T_n^\circ$, $n\geq2$
has the property of \textit{sampling consistency}, if when we select a
leaf uniformly and delete it (together with the adjacent branch point if its degree is reduced to 2), 
then the new tree, denoted by $T_{n,-1}^\circ$, is distributed as
$T_{n-1}^\circ$. 
Denote by $d:\bD_n\rightarrow\bD_{n-1}$ the induced deletion operator on the space
$\bD_n$ of probability measures on $\bT_n^\circ$, so that for the distribution $P_n$ of $T_n^\circ$, we
define $d(P_{n})$ as the distribution of $T_{n,-1}^\circ$. Sampling
consistency is equivalent to $d(P_n)=P_{n-1}$. This property is also
called \textit{deletion stability} in \cite{For-05}. 

\begin{prop}
\label{samplecons} The unlabelled alpha-gamma trees for $0\le\alpha\le 1$ and $0\le\gamma\le\alpha$ are sampling
consistent.
\end{prop}

\begin{proof} The sampling consistency formula $(14)$ in \cite{HMPW} states that $d(P_n)=P_{n-1}$ is equivalent to
\begin{eqnarray}\label{del3}
   q(n_1,\ldots,n_k)&=&\sum_{i=1}^{k}\frac{(n_i+1)(m_{n_i+1}+1)}{nm_{n_i}}q((n_1,\ldots,n_i+1,\ldots,n_k)^\downarrow)\nonumber\\
   &&+\frac{m_1+1}{n}q(n_1,\ldots,n_k,1)+\frac{1}{n}q(n-1,1)q(n_1,\ldots,n_k)
\end{eqnarray}
for all $n_1\ge\ldots\ge n_k\ge 1$ with $n_1+\ldots+n_k=n-1$, where $m_j$ is the number of $n_i$, $i\in[k]$, that equal $j$,
and where $q$ is the splitting rule of $T_n^\circ\sim P_n$. In terms of EPPFs (\ref{spliteppf}), formula (\ref{del3}) is
equivalent to 
\begin{equation}\label{del4}
\left(1-p(n-1,1)\right)p(n_1,\ldots,n_k)=\sum_{i=1}^kp((n_1,\ldots,n_i+1,\ldots,n_k)^\downarrow)+p(n_1,\ldots,n_k,1).
\end{equation}

Now according to Proposition \ref{prop1}, the EPPF of the
alpha-gamma model with $\alpha<1$ is
\begin{equation}
p_{\alpha,\gamma}^{\rm
seq}(n_1,\ldots,n_k)=\frac{Z_n}{\Gamma_\alpha(n)}\left(\gamma+(1-\alpha-\gamma)\frac{1}{(n-1)(n-2)}\sum_{u\neq
v}n_un_v\right)p_{\alpha,-\alpha-\gamma}^{\rm
PD^*}(n_1,\ldots,n_k),
\end{equation}
where $\Gamma_\alpha(n)=\Gamma(n-\alpha)/\Gamma(1-\alpha)$. Therefore, $p_{\alpha,\gamma}^{\rm seq}(n_1,\ldots,n_i+1,\ldots,n_k)$ can be written as
\begin{eqnarray} 
&&\hspace{-0.5cm}\left (p_{\alpha,\gamma}^{\rm
seq}(n_1,\ldots,n_k)+2(1-\alpha-\gamma)\frac{(n-2)(n-1-n_i)-\sum_{u\neq
v}n_u
n_v}{n(n-1)(n-2)}\frac{Z_n}{\Gamma_\alpha(n)}p_{\alpha,-\alpha-\gamma}^{\rm
PD^*}(n_1,\ldots,n_k)\right )\nonumber\\
&&\hspace{0.5cm}\times\frac{n_i-\alpha}{n-1-\alpha}\nonumber
\end{eqnarray}
and $p_{\alpha,\gamma}^{\rm seq}(n_1,\ldots,n_k,1)$ as
\begin{eqnarray}
&&\hspace{-0.5cm}\left (p_{\alpha,\gamma}^{\rm
seq}(n_1,\ldots,n_k)+2(1-\alpha-\gamma)\frac{(n-1)(n-2)-\sum_{u\neq
v}n_u
n_v}{n(n-1)(n-2)}\frac{Z_n}{\Gamma_\alpha(n)}p_{\alpha,-\alpha-\gamma}^{\rm
PD^*}(n_1,\ldots,n_k)\right )\nonumber\\
&&\hspace{0.5cm}\times\frac{(k-1)\alpha-\gamma}{n-1-\alpha}.\nonumber
\end{eqnarray}
Sum over the above formulas, then the right-hand side of
(\ref{del4}) is
\begin{equation}
\left
(1-\frac{1}{n-1-\alpha}\left(\gamma+\frac{2}{n}(1-\alpha-\gamma)\right
)\right)p_{\alpha,\gamma}^{\rm seq}(n_1,\ldots,n_k).\nonumber
\end{equation}
Notice that $p_{\alpha,\gamma}^{\rm
seq}(n-1,1)=\left(\gamma+2(1-\alpha-\gamma)/ n \right
)/(n-1-\alpha)$. Hence, the splitting rules of the alpha-gamma model
satisfy (\ref{del4}), which implies sampling consistency for $\alpha<1$. The case $\alpha=1$ is postponed to Section \ref{sectalpha1}.
\end{proof}

Moreover, sampling consistency can be enhanced to \em strong sampling
consistency \em \cite{HMPW} by requiring that $(T_{n-1}^\circ,T_{n}^\circ)$ has
the same distribution as $(T_{n,-1}^\circ,T_{n}^\circ)$. 

\begin{prop}\label{prop13}
\label{ssc} The alpha-gamma model is strongly sampling consistent if
and only if $\gamma=1-\alpha$.
\end{prop}
\begin{proof} For $\gamma=1-\alpha$, the model is known to be strongly sampling consistent, cf. Section \ref{sectstable}.
\setlength{\unitlength}{0.5cm}
\begin{picture}(30,5.5)
\multiput(10,1)(10,0){2}{\line(0,1){2.8}} \multiput(10,
2.4)(10,0){2}{\line(1,1){1}} \multiput(10,
3.8)(10,0){2}{\line(1,1){1}} \multiput(10,
3.8)(10,0){2}{\line(-1,1){1}} \put(20,2.4){\line(-1,1){1}}
 \put(9.6,0){
$\mathbf{t}_3^\circ$} \put(19.7,0){$\mathbf{t}_4^\circ$}
\end{picture}

If $\gamma\neq 1-\alpha,$ consider the above two deterministic
unlabelled trees.
$$\bP(T_4^\circ=\ft_4^\circ)=q_{\alpha,\gamma}^{\rm
seq}(2,1,1)q_{\alpha,\gamma}^{\rm
seq}(1,1)=(\alpha-\gamma)(5-5\alpha+\gamma)/((2-\alpha)(3-\alpha)).$$
Then we delete one of the two leaves at the first branch point of
$\ft_4^\circ$ to get $\ft_3^\circ$. Therefore
$$\bP((T_{4,-1}^\circ,T_4^\circ)=(\ft_3^\circ,\ft_4^\circ))=\frac{1}{2}\bP(T_4^\circ=\ft_4^\circ)
=\frac{(\alpha-\gamma)(5-5\alpha+\gamma)}{2(2-\alpha)(3-\alpha)}.$$
On the other hand, if $T_3^\circ=\ft_3^\circ$, we have to add the
new leaf to the first branch point to get $\ft_4^\circ$. Thus
$$\bP((T_3^\circ,T_4^\circ)=(\ft_3^\circ,\ft_4^\circ))=\frac{\alpha-\gamma}{3-\alpha}\bP(T_3^\circ=\ft_3^\circ)=
\frac{(\alpha-\gamma)(2-2\alpha+\gamma)}{(2-\alpha)(3-\alpha)}.$$ It
is easy to check that
$\bP((T_{4,-1}^\circ,T_4^\circ)=(\ft_3^\circ,\ft_4^\circ))\neq
\bP((T_3^\circ,T_4^\circ)=(\ft_3^\circ,\ft_4^\circ))$ if $\gamma\neq
1-\alpha$, which means that the alpha-gamma model is then not strongly
sampling consistent.
\end{proof}

\section{Dislocation measures and asymptotics of alpha-gamma trees}
\subsection{Dislocation measures associated with the alpha-gamma-splitting rules}

Theorem \ref{thm2} claims that the alpha-gamma trees are sampling consistent, which we proved in 
Section \ref{sectsc}, and identifies the integral representation of the splitting rule in terms of
a dislocation measure, which we will now establish.

\begin{proof}[Proof of Theorem \ref{thm2}] Firstly, we make some rearrangement for the coefficient of the
sampling consistent splitting rules of alpha-gamma trees identified in Proposition \ref{prop10}:
\begin{eqnarray}
&&\hspace{-0.5cm}\gamma+(1-\alpha-\gamma)\frac{1}{n(n-1)}\sum_{i\neq j}n_in_j\nonumber\\
&&=\frac{(n+1-\alpha-\gamma)(n-\alpha-\gamma)}{n(n-1)}\left(\gamma+(1-\alpha-\gamma)\left(\sum_{i\neq j}A_{ij}
+2\sum_{i=1}^kB_i+C\right)\right),\nonumber
\end{eqnarray}
where\vspace{-0.1cm}
\begin{eqnarray}
A_{ij}&=&\frac{(n_i-\alpha)(n_j-\alpha)}{(n+1-\alpha-\gamma)(n-\alpha-\gamma)},\nonumber\\
B_i&=&\frac{(n_i-\alpha)((k-1)\alpha-\gamma)}{(n+1-\alpha-\gamma)(n-\alpha-\gamma)},\nonumber\\
C&=&\frac{((k-1)\alpha-\gamma)(k\alpha-\gamma)}{(n+1-\alpha-\gamma)(n-\alpha-\gamma)}.\nonumber
\end{eqnarray}
Notice that $B_ip_{\alpha,-\alpha-\gamma}^{\rm PD^*}(n_1,\ldots,n_k)$ simplifies to
\begin{eqnarray}
&&\hspace{-0.5cm}\frac{(n_i-\alpha)((k-1)\alpha-\gamma)}{(n+1-\alpha-\gamma)(n-\alpha-\gamma)}\frac{\alpha^{k-2}\Gamma(k-1-\gamma/\alpha)}{Z_n\Gamma(1-\gamma/\alpha)}\Gamma_\alpha(n_1)\ldots\Gamma_\alpha(n_k)\nonumber\\
&&=\frac{Z_{n+2}}{Z_n(n+1-\alpha-\gamma)(n-\alpha-\gamma)}\frac{\alpha^{k-1}\Gamma(k-\gamma/\alpha)}{Z_{n+2}\Gamma(1-\gamma/\alpha)}\Gamma_\alpha(n_1)\ldots\Gamma_\alpha(n_i+1)\ldots\Gamma_\alpha(n_k)\nonumber\\
&&=\frac{\widetilde{Z}_{n+2}}{\widetilde{Z}_n}p_{\alpha,-\alpha-\gamma}^{\rm
PD^*}(n_1,\ldots,n_i+1,\ldots,n_k,1),\nonumber
\end{eqnarray}
where $\Gamma_\alpha(n)=\Gamma(n-\alpha)/\Gamma(1-\alpha)$ and 
$\widetilde{Z}_n=Z_n\alpha\Gamma(1-\gamma/\alpha)/\Gamma(n-\alpha-\gamma)$ is the normalisation constant in (\ref{Kingman})
for $\nu={\rm PD}^*_{\alpha,-\gamma-\alpha}$, as can be read from \cite[Formula (17)]{HPW}.
According to (\ref{Kingman}),
\begin{equation}
p_{\alpha,-\alpha-\gamma}^{\rm
PD^*}(n_1,\ldots,n_k)=\frac{1}{\widetilde{Z}_n}\int_{\cS^\downarrow}\sum_{\substack{i_1,\ldots,i_k\geq
1\\ {\rm distinct}}}\prod_{l=1}^k s_{i_l}^{n_l} {\rm
PD}^*_{\alpha,-\alpha-\gamma}(ds).\nonumber
\end{equation}
Thus,
\begin{eqnarray}
B_ip_{\alpha,-\alpha-\gamma}^{\rm PD^*}(n_1,\ldots,n_k)&=&
\frac{1}{\widetilde{Z}_n}\int_{\cS^\downarrow}\sum_{\substack{i_1,\ldots,i_k\geq 1\\
{\rm distinct}}}\prod_{l=1}^k s_{i_l}^{n_l}
\left(\sum_{u\in\{i_1,\ldots,i_k\},v\not\in
\{i_1,\ldots,i_k\}}s_us_v\right)
{\rm PD}^*_{\alpha,-\alpha-\gamma}(ds)\nonumber
\end{eqnarray}
 Similarly,
\begin{eqnarray}
A_{ij}p_{\alpha,-\alpha-\gamma}^{\rm
PD^*}(n_1,\ldots,n_k)&=&\frac{1}{\widetilde{Z}_n}\int_{\cS^\downarrow}\sum_{\substack{i_1,\ldots,i_k\geq 1\\
{\rm distinct}}}\prod_{l=1}^k s_{i_l}^{n_j}
\left(\sum_{u,v\in\{i_1,\ldots,i_k\}:u\neq v}s_us_v\right)
 {\rm PD}^*_{\alpha,-\alpha-\gamma}(ds)\nonumber\\
Cp_{\alpha,-\alpha-\gamma}^{\rm
PD^*}(n_1,\ldots,n_k)&=&\frac{1}{\widetilde{Z}_n}\int_{\cS^\downarrow}\sum_{\substack{i_1,\ldots,i_k\geq 1\\
{\rm distinct}}}\prod_{l=1}^k s_{i_l}^{n_l}
\left(\sum_{u,v\not\in\{i_1,\ldots,i_k\}:u\neq v}s_us_v\right)
{\rm PD}^*_{\alpha,-\alpha-\gamma}(ds),\nonumber
\end{eqnarray}
Hence,
the EPPF $p_{\alpha,\gamma}^{\rm seq}(n_1,\ldots,n_k)$ of the sampling consistent splitting rule takes the following
form:
\begin{eqnarray}\label{nuag3}
&&\hspace{-0.5cm}\frac{(n+1-\alpha-\gamma)(n-\alpha-\gamma)Z_n}{n(n-1)\Gamma_\alpha(n)}\left(\gamma+(1-\alpha-\gamma)\left(\sum_{i\neq j}A_{ij}
+2\sum_{i=1}^kB_i+C\right)\right)p_{\alpha,\gamma}^{\rm PD^*}(n_1,\ldots,n_k)\nonumber\\ 
&&=\frac{1}{Y_n}\int_{\cS^\downarrow}\sum_{\substack{i_1,\ldots ,i_k\geq
1\\ {\rm distinct}}}\prod_{j=1}^k
s_{i_j}^{n_j}\left(\gamma+(1-\alpha-\gamma)\sum_{i\neq j}
s_is_j\right){\rm PD}^*_{\alpha,-\alpha-\gamma}(ds),
\end{eqnarray}
where
$Y_n=n(n-1)\Gamma_\alpha(n)\alpha\Gamma(1-\gamma/\alpha)/\Gamma(n+2-\alpha-\gamma)$
is the normalization constant. Hence, we have
$\nu_{\alpha,\gamma}(ds)=\Big(\gamma+(1-\alpha-\gamma)\sum_{i\neq j}
s_is_j\Big){\rm PD^*}_{\alpha,-\alpha-\gamma}(ds)$.
\end{proof}

\subsection{The alpha-gamma model when $\alpha=1$, spine with bushes of singleton-trees}\label{sectalpha1}

Within the discussion of the alpha-gamma model so far, we restricted to $0\leq\alpha<1$. 
In fact, we can still get some interesting results when $\alpha=1$. 
The weight of each leaf edge is $1-\alpha$ in the growth
procedure of the alpha-gamma model. If $\alpha=1$, the weight of each
leaf edge becomes zero, which means that the new leaf can only be
inserted to internal edges or branch points. Starting from the two
leaf tree, leaf 3 must be inserted into the root edge or the branch
point. Similarly, any new leaf must be inserted into the spine
leading from the root to the common ancestor of leaf 1 and leaf 2. 
Hence, the shape of the tree is just a spine with some
bushes of one-leaf subtrees rooted on it. Moreover, the first split
of an $n$-leaf tree will be $(n-k+1,1,\ldots,1)$ for some $2\leq k\leq
n-1$. The cases $\gamma=0$ and $\gamma=1$ lead to degenerate trees
 with, 
respectively, all leaves connected to a single branch point and all leaves 
connected to a spine of binary branch points (comb).

\begin{prop}\label{alpha1}
Consider the alpha-gamma model with $\alpha=1$ and $0<\gamma<1$.
\begin{enumerate}\item[\rm(a)] The model is sampling consistent with splitting rules 
   \begin{eqnarray}\label{a11}
   \hspace{-0.5cm}&&\hspace{-0.5cm}q_{1,\gamma}^{\rm seq}(n_1,\ldots,n_k)\nonumber\\
   \hspace{-0.5cm}&&=\begin{cases}
   \gamma\Gamma_\gamma(k-1)/(k-1)!, &\text{if}\ 2\leq k\leq n-1\ \text{and}\ (n_1,\ldots,n_k)=(n-k+1,1,\ldots,1);\\
   \Gamma_\gamma(n-1)/(n-2)!, &\text{if}\ k=n\ \text{and}\
   (n_1,\ldots,n_k)=(1,\ldots,1);\\
   0,&\text{otherwise},
   \end{cases}
   \end{eqnarray}
where $n_1\geq\ldots \geq n_k\geq 1$ and $n_1+\ldots +n_k=n$.
  \item[\rm(b)] The dislocation measure associated with the splitting
  rules can be expressed as follows
  \begin{equation}
  \int_{\cS^\downarrow}f(s_1,0,\ldots)\nu_{1,\gamma}(ds)=\int_0^1
  f(s_1,0,\ldots)\left(\gamma(1-s_1)^{-1-\gamma}ds_1+\delta_0(ds_1)\right).
  \end{equation}
  In particular, it does \em not \em satisfy $\nu(\{s\in\cS^\downarrow:s_1+s_2+\ldots<1\})=0$. 
\end{enumerate}
\end{prop}

\begin{proof}(a) We start from the growth procedure of the
alpha-gamma model when $\alpha=1$. Consider a first split into
$(n-k+1,1,\ldots,1)$ for some labelled $n$-leaf tree. Suppose its
first branch point is created when the leaf $l$ is inserted to the
root edge for $l\geq 3$. At this time the first split is $(l-1,1)$
with a probability $\gamma/(l-2)$ as $\alpha=1$. In the following
insertion, leaves $l+1,\ldots,n$ have been added either to the first
branch point or to the subtree with $l-1$ leaves at this time. Hence
the probability that the first split of this tree is
$(n-k+1,1,\ldots,1)$ is 
$$\frac{(n-k-1)!}{(n-2)!}\gamma\Gamma_\gamma(k-1),$$ which does not depend
on $l$. Notice that the growth rules imply that if the first split is
$(n-k+1, 1,\ldots,1)$ with $k\le n-1$, then leaves $1$ and $2$ will be 
located in the subtree with $n-k+1$ leaves. There are ${n-2\choose n-k-1}$ 
labelled trees with the above first split. Therefore,
$$q_{1,\gamma}^{\rm seq}(n-k+1, 1,\ldots,1)={n-2\choose
n-k-1}\frac{(n-k-1)!}{(n-2)!}\gamma\Gamma_\gamma(k-1)=\gamma\Gamma_\gamma(k-1)/(k-1)!.$$
On the other hand, there is only one $n$-leaf labelled tree with a
first split $(1,\ldots,1)$ and in this case, all leaves have been added
to the only branch point
. Hence
$$q_{1,\gamma}^{\rm seq}(1,\ldots,1)=\Gamma_\gamma(n-1)/(n-2)!.$$
For sampling consistency, we check criterion (\ref{del3}), which reduces to the two formulas
\begin{eqnarray}
(1-q_{1,\gamma}^{\rm seq}(n-1,1))q_{1,\gamma}^{\rm seq}(n-k,
1,\ldots,1)&=&\frac{k}{n}q_{1,\gamma}^{\rm seq}(n-k,
1,\ldots,1)\nonumber\\&&+\frac{n-k+1}{n}q_{1,\gamma}^{\rm
seq}(n-k+1,
1,\ldots,1)\nonumber\\
(1-q_{1,\gamma}^{\rm seq}(n-1,1))q_{1,\gamma}^{\rm seq}(
1,\ldots,1)&=&\frac{2}{n}q_{1,\gamma}^{\rm seq}(2,
1,\ldots,1)+q_{1,\gamma}^{\rm seq}(1,\ldots,1).\nonumber
\end{eqnarray}

(b) According to (\ref{a11}),
\begin{eqnarray}\label{a12}
&&\hspace{-0.5cm}q_{1,\gamma}^{\rm seq}(n-k+1,1,\ldots,1)\nonumber\\
&&={n\choose
n-k+1}\frac{\Gamma_\gamma(n+1)}{n!}\gamma
B(n-k+2,k-1-\gamma)\nonumber\\
&&=\frac{1}{Y_n}{n\choose n-k+1}\int_0^1
s_1^{n-k+1}(1-s_1)^{k-1}\left(\gamma(1-s_1)^{-1-\gamma}ds_1\right)\nonumber\\
&&=\frac{1}{Y_n}{n\choose n-k+1}\int_0^1
s_1^{n-k+1}(1-s_1)^{k-1}\left((\gamma(1-s_1)^{-1-\gamma}ds_1+\delta_0(ds_1)\right),
\end{eqnarray}
where $Y_n=n!/\Gamma_\gamma(n+1)$. Similarly,
\begin{equation}\label{a13}
 q_{1,\gamma}^{\rm
seq}(1,\ldots,1)=\frac{1}{Y_n}\int_0^1
\left(n(1-s_1)^{n-1}s_1+(1-s_1)^n\right)\left((\gamma(1-s_1)^{-1-\gamma}ds_1+\delta_0(ds_1)\right).
\end{equation}
Formulas (\ref{a12}) and (\ref{a13}) are of the form of \cite[Formula (2)]{HMPW}, which generalises (\ref{Kingman}) to
the case where $\nu$ does not necessarily satisfy $\nu(\{s\in\cS^\downarrow:s_1+s_2+\ldots<1\})=0$, hence 
$\nu_{1,\gamma}$ is identified.
\end{proof}

\subsection{Continuum random trees and self-similar trees}\label{seccrts}
Let $B\subset \mathbb{N}$ finite. A \em labelled tree with edge lengths \em is a pair
$\vartheta=(\mathbf{t},\eta)$, where
$\mathbf{t}\in\mathbb{T}_B$ is a labelled tree, $\eta=(\eta_A,A\in \mathbf{t}\setminus\{\mbox{\sc root}\})$ is a collection of marks, 
and every edge $C\rightarrow A$ of $\ft$ is associated with mark $\eta_A\in(0,\infty)$
at vertex $A$, which we interpret as the \em edge length \em of $C\rightarrow A$. 
Let $\Theta_B$ be the set of such trees $(\ft,\eta)$ with $\ft\in\mathbb{T}_B$.

We now introduce continuum trees, following the construction
by Evans et al. in \cite{EPW}. A complete separable metric space
$(\tau,d)$ is called an $\mathbb{R}$-tree, if it satisfies
the following two conditions:
\begin{enumerate}
   \item for all $x,y\in \tau$, there is an isometry $\varphi_{x,y}:[0,d(x,y)]\rightarrow
   \tau$ such that $\varphi_{x,y}(0)=x$ and
   $\varphi_{x,y}(d(x,y))=y$,
   \item for every injective path $c:[0,1]\rightarrow \tau$ with
   $c(0)=x$ and  $c(1)=y$, one has $c([0,1])=\varphi_{x,y}([0,d(x,y)])$.
\end{enumerate}
We will consider rooted $\mathbb{R}$-trees $(\tau,d,\rho)$, where $\rho\in\tau$ is
a distinguished element, the \em root\em. We think of the root as the lowest element of the tree.\pagebreak[2]

We denote the range of $\varphi_{x,y}$ by
$[[x,y]]$ and call the
 quantity $d(\rho,x)$ the \em height \em of $x$. We say that $x$ is an
 ancestor of $y$ whenever $x\in [[\rho,y]]$. We let $x\wedge
 y$ be
 the unique element in $\tau$ such that
 $[[\rho,x]]\cap[[\rho,y]]=[[\rho,x\wedge y]]$, and call it the
 \em highest common ancestor \em of $x$ and $y$ in $\tau$. Denoted by
 $(\tau_x,d|_{\tau_x},x)$ the set of $y\in\tau$ such that $x$ is 
 an ancestor of $y$, which is an $\mathbb{R}$-tree  rooted at $x$ that
we call the \em fringe subtree \em of $\tau$ above $x$.

Two rooted $\mathbb{R}$-trees $(\tau,d,\rho),(\tau^\prime,d^\prime,\rho^\prime)$ are
called \em equivalent \em if there is a bijective isometry between the two
metric spaces that maps the root of one to the root of the other. We
also denote by $\Theta$ the set of equivalence classes of compact
rooted $\mathbb{R}$-trees. We define the \em Gromov-Hausdorff distance \em between two
rooted $\bR$-trees (or their equivalence classes) as 
\beqs d_{\rm GH}(\tau,\tau^\prime)=\inf\{d_{\rm H}(\widetilde{\tau},\widetilde{\tau}^\prime)\}
\eeqs
where the infimum is over all metric spaces $E$ and isometric embeddings $\widetilde{\tau}\subset E$ of
$\tau$ and $\widetilde{\tau}^\prime\subset E$ of $\tau^\prime$ with common root $\widetilde{\rho}\in E$; the Hausdorff distance on compact 
subsets of $E$ is denoted by $d_{\rm H}$. Evans et al. \cite{EPW} showed that $(\Theta,d_{\rm GH})$ is 
a complete separable metric space.

We call an element $x\in \tau$, $x\neq \rho$, in a rooted
$\mathbb{R}$-tree $\tau$, a \em leaf \em if its  removal does not disconnect
$\tau$, and let $\mathcal{L}(\tau)$ be the set of leaves of $\tau$.
On the other hand, we call an element of $\tau$ a \em branch point\em, if
it has the form $x\wedge y$ where $x$ is neither an ancestor of $y$
nor vice-visa. Equivalently, we can define branch points as points
disconnecting $\tau$ into three or more connected
components when removed. We let $\mathcal{B}(\tau)$ be the set of
branch points of $\tau$.

A \em weighted $\bR$-tree \em $(\tau,\mu)$ is called a \em continuum tree \em \cite{Ald-91}, if $\mu$ 
is a probability measure on $\tau$ and
\begin{enumerate}
   \item $\mu$ is supported by the set $\mathcal{L}(\tau)$,
   \item $\mu$ has no atom,
   \item for every $x\in\tau\backslash\mathcal{L}(\tau)$,
   $\mu(\tau_x)>0$.
\end{enumerate}
A \em continuum random tree (CRT) \em is a random variable whose values are
continuum trees, defined on some probability space
$(\Omega,\mathcal{A},\mathbb{P})$. Several methods to formalize 
this have been developed \cite{Ald-CRT3,EW,GPW}. For technical simplicity, 
we use the method of Aldous \cite{Ald-CRT3}. Let the space
$\ell_1=\ell_1(\mathbb{N})$ be the base space for defining CRTs. We
endow the set of compact subsets of $\ell_1$ with the Hausdorff
metric, and the set of probability measures on $\ell_1$ with any
metric inducing the topology of weak convergence, so that the set of
pairs $(T,\mu)$ where $T$ is a rooted $\mathbb{R}$-tree embedded as
a subset of $\ell_1$ and $\mu$ is a measure on $T$, is endowed with
the product $\sigma$-algebra.


An exchangeable \em $\mathcal{P}_\bN$-valued fragmentation process \em 
$(\Pi(t),t\geq0)$ is called \em self-similar \em with index 
$a\in \mathbb{R}$ if given $\Pi(t)=\pi=\{\pi_i,i\ge 1\}$ with
asymptotic frequencies $|\pi_i|=\lim_{n\rightarrow\infty}n^{-1}\#[n]\cap\pi_j$, 
the random variable $\Pi(t+s)$ has the same law as the random partition
whose blocks are those of $\pi_i\cap\Pi^{(i)}(|\pi_i|^a s),i\geq 1$,
where $(\Pi^{(i)}, i\geq1)$ is a sequence of i.i.d. copies of
$(\Pi(t),t\geq 0)$. The process 
$(|\Pi(t)|^\downarrow,t\geq 0)$ is an \em $S^\downarrow$-valued 
self-similar fragmentation process\em. Bertoin \cite{Ber-ss} proved that the distribution of a $\mathcal{P}_\bN$-valued
self-similar fragmentation process is determined by a triple $(a,c,\nu)$,
where $a\in\bR$, $c\ge 0$ and $\nu$ is a dislocation measure on 
$S^\downarrow$. For this article, we are only interested in 
the case $c=0$ and when $\nu(s_1+s_2+\ldots<1)=0$. We call $(a,\nu)$ the 
characteristic pair. When $a=0$, the process $(\Pi(t),t\ge 0)$ is also called \em homogeneous fragmentation process\em.

A CRT $(\mathcal{T},\mu)$ is a \em self-similar CRT \em
with index $a=-\gamma<0$ if for every $t\geq 0$, given
$(\mu(\mathcal{T}_t^i),i\geq 1))$ where $\mathcal{T}_t^i, i\geq 1$
is the ranked order of connected components of the open set
$\{x\in\tau: d(x,\rho(\tau))>t\}$, the continuum random trees
$$
   \left
   (\mu(\mathcal{T}^1_t)^{-\gamma}\mathcal{T}_t^1,\frac{\mu(\cdot
   \cap\mathcal{T}_t^1)}{\mu(\mathcal{T}^1_t)}\right ),  \left
   (\mu(\mathcal{T}^2_t)^{-\gamma}\mathcal{T}_t^2,\frac{\mu(\cdot
   \cap\mathcal{T}_t^2)}{\mu(\mathcal{T}^2_t)}\right ),\ldots 
$$
are i.i.d copies of $(\mathcal{T},\mu)$, where
$\mu(\mathcal{T}^i_t)^{-\gamma}\mathcal{T}_t^i$ is the tree that has
the same set of points as $\mathcal{T}^i_t$, but whose distance function is
divided by $\mu(\mathcal{T}^i_t)^{\gamma}$.  Haas and Miermont in
\cite{HM} have shown that there exists a self-similar continuum
random tree $\cT_{(\gamma,\nu)}$ characterized by such a pair 
$(\gamma,\nu)$, which can
be constructed from a self-similar fragmentation
process with characteristic pair $(\gamma,\nu)$.

\subsection{The alpha-gamma model when $\gamma=1-\alpha$, sampling from the stable CRT}\label{sectstable}

Let $(\cT,\rho,\mu)$ be the stable tree of Duquesne and Le Gall \cite{DuL-02}. The distribution on $\Theta$ of any CRT is 
determined by its so-called finite-dimensional marginals: the distributions of $\cR_k$, $k\ge 1$, the subtrees $\cR_k\subset\cT$ 
defined as the discrete trees with edge lengths spanned by $\rho,U_1,\ldots,U_k$, where given $(\cT,\mu)$, the sequence 
$U_i\in\cT$, $i\ge 1$, of leaves is sampled independently from $\mu$. See also \cite{Mie-05,DuL-05,HMPW,HPW,Mar-08} for various
approaches to stable trees. Let us denote the discrete tree without edge lengths associated with $\cR_k$ by $T_k$ and note the Markov branching structure.

\begin{lemm}[Corollary 22 in \cite{HMPW}] Let $1/\alpha\in(1,2]$. The trees $T_n$, $n\ge 1$, sampled from the $(1/\alpha)$-stable
 CRT are Markov branching trees, 
  whose splitting rule has EPPF 
  \beqs
  p^{\rm stable}_{1/\alpha}(n_1,\ldots,n_k)=\frac{\alpha^{k-2}\Gamma(k-1/\alpha)\Gamma(2-\alpha)}{\Gamma(2-1/\alpha)\Gamma(n-\alpha)}
\prod_{j=1}^k\frac{\Gamma(n_j-\alpha)}{\Gamma(1-\alpha)}
  \eeqs
  for any $k\ge 2$, $n_1\ge1,\ldots,n_k\ge 1$, $n=n_1,\ldots,n_k$.
\end{lemm}

We recognise $p^{\rm stable}_{1/\alpha}=p^{\rm PD^*}_{\alpha,-1}$ in (\ref{EPmod}), and by Proposition \ref{prop1}, we have $p^{\rm PD^*}_{\alpha,-1}=p^{\rm seq}_{\alpha,1-\alpha}$. This observation yields the following corollary:

\begin{coro} The alpha-gamma trees with $\gamma=1-\alpha$ are strongly sampling consistent and exchangeable.
\end{coro}
\begin{proof} These properties follow from the representation by sampling from the stable CRT, particularly the exchangeability of 
  the sequence $U_i$, $i\ge 1$. Specifically, since $U_i$, $i\ge 1$, are conditionally independent
  and identically distributed given $(\cT,\mu)$, they are exchangeable. If we denote by   
  $\cL_{n,-1}$ the random set of leaves $\cL_n=\{U_1,\ldots,U_n\}$ with a uniformly chosen 
  member removed, then $(\cL_{n,-1},\cL_n)$ has the same conditional distribution as 
  $(\cL_{n-1},\cL_n)$. Hence the pairs of (unlabelled) tree shapes spanned by $\rho$ and these 
  sets of leaves have the same distribution -- this is strong sampling consistency as defined before Proposition \ref{prop13}. 
\end{proof}

\subsection{Dislocation measures in size-biased order} In actual
calculations, we may find that the splitting rules in Proposition
\ref{prop1} are quite difficult and the corresponding dislocation
measure $\nu$ is always inexplicit, which leads us to transform
$\nu$ to a more explicit form. The method proposed here is to change 
the space $\cS^\downarrow$ into the space $[0,1]^\mathbb{N}$ and to
rearrange the elements $s\in\cS^\downarrow$ under $\nu$ into the \em size-biased random order \em
that places $s_{i_1}$ first with probability $s_{i_1}$ (its \em size\em) and, successively, the remaining ones with probabilities $s_{i_j}/(1-s_{i_1}-\ldots-s_{i_{j-1}})$ proportional to 
their sizes $s_{i_j}$ into the following positions, $j\ge 2$.

\begin{defi}\rm We call a measure $\nu^{\rm sb}$ on the space $[0,1]^\mathbb{N}$ the
size-biased dislocation measure associated with dislocation measure
$\nu$, if for any subset $A_1\times A_2\times\ldots\times A_k\times
[0,1]^\mathbb{N}$ of $[0,1]^\mathbb{N}$,
\begin{equation}
\nu^{\rm sb}(A_1\times A_2\times\ldots\times A_k\times
[0,1]^\mathbb{N})=\sum_{\substack{i_1,\ldots,i_k\ge 1\\
   {\rm distinct}}}\int_{\{s\in
\cS^\downarrow:s_{i_1}\in A_1,\ldots,s_{i_k}\in
A_k\}}\frac{s_{i_1}\ldots s_{i_k}}{\prod_{j=1}^{k-1}(1-\sum_{l=1}^j
  s_{i_l})} \nu(ds)\label{sizebias1}
\end{equation}
for any $k\in\bN$, where $\nu$ is a dislocation measure on $\cS^\downarrow$ satisfying
$\nu(s\in\cS^\downarrow:s_1+s_2+\ldots<1)=0$. We also denote by $\nu_k^{\rm sb}(A_1\times A_2\times\ldots\times A_k)=\nu^{\rm sb}(A_1\times
A_2\times\ldots\times A_k\times [0,1]^\mathbb{N})$  the distribution of the 
first $k$ marginals.
\end{defi}
The sum in (\ref{sizebias1}) is over all possible rank sequences $(i_1,\ldots,i_k)$ to determine the first $k$ 
entries of the size-biased vector. The integral in (\ref{sizebias1}) is over the decreasing sequences that have the $j$th entry of the re-ordered vector fall into $A_j$, $j\in[k]$. 
Notice that the support of such a size-biased dislocation measure
$\nu^{\rm sb}$ is a subset of $\cS^{\rm sb}:=\{ s\in
[0,1]^\mathbb{N}: \sum_{i=1}^\infty s_i=1 \}$. If we
denote by $s^\downarrow$ the sequence $s\in\cS^{\rm sb}$ rearranged into ranked order,
taking (\ref{sizebias1}) into formula (\ref{Kingman}), we obtain

\begin{prop}\label{sizebias} The EPPF associated with a dislocation measure $\nu$ can be represented as:
$$
  p(n_1,\ldots,n_k)=\frac{1}{\widetilde{Z}_n}\int_{[0,1]^k}x_1^{n_1-1}\ldots x_k^{n_k-1}\prod_{j=1}^{k-1}(1-\sum_{l=1}^jx_l)
  \nu_k^{\rm sb}(dx),$$
 where $
   \nu^{\rm sb}$
 is the size-biased dislocation measure associated with $\nu$, where $n_1\geq\ldots\geq n_k\geq 1, k\geq 2, n=n_1+\ldots+n_k$ and
 $x=(x_1,\ldots,x_k)$.
\end{prop}

Now turn to see the case of Poisson-Dirichlet measures ${\rm PD}^*_{\alpha,\theta}$ to then study $\nu_{\alpha,\gamma}^{\rm sb}$.

\begin{lemm}\label{PDSB}
If we define ${\rm GEM}^*_{\alpha,\theta}$ as the size-biased
dislocation measure associated with ${\rm PD}_{\alpha,\theta}^*$ for
$0<\alpha<1$ and $\theta>-2\alpha$, then the first $k$ marginals have joint density
\begin{equation}\label{GEM}
{\rm gem}_{\alpha,\theta}^*(x_1,\ldots,x_k)=\frac{\alpha\Gamma(2+\theta/\alpha)}{\Gamma(1-\alpha)\Gamma(\theta+\alpha+1)\prod_{j=2}^kB(1-\alpha,\theta+j\alpha)}\frac{(1-\sum_{i=1}^kx_i)^{\theta+k\alpha}\prod_{j=1}^k
x_j^{-\alpha}}{\prod_{j=1}^k(1-\sum_{i=1}^j x_i)},
\end{equation}
where $B(a,b)=\int_0^1x^{a-1}(1-x)^{b-1}dx$
is the beta function.
\end{lemm}

This is a simple $\sigma$-finite extension of the {\rm GEM} distribution and
(\ref{GEM}) can be derived analogously to Lemma \ref{crp1}. Applying
Proposition \ref{sizebias}, we can get an explicit form of the
size-biased dislocation measure associated with the alpha-gamma
model.	

\begin{proof}[Proof of Proposition \ref{prop4}] We start our proof
from the dislocation measure associated with the alpha-gamma model.
According to (\ref{thm2nu}) and (\ref{sizebias1}), the first $k$ marginals of
$\nu_{\alpha,\gamma}^{\rm sb}$ are given by
\begin{eqnarray*}
&&\hspace{-0.5cm}\nu_k^{\rm sb}(A_1\times\ldots\times A_k)\\
&&=\sum_{\substack{i_1,\ldots,i_k\ge 1\\
   {\rm distinct}}}\int_{\{s\in
\cS^\downarrow:s_{i_j}\in A_j,j\in[k]\}}\frac{s_{i_1}\ldots s_{i_k}}{\prod_{j=1}^{k-1}(1-\sum_{l=1}^j
  s_{i_l})}
\left (\gamma+(1-\alpha-\gamma)\sum_{i\neq j}s_i s_j \right){\rm
PD}^*_{\alpha,-\alpha-\gamma}(ds)\nonumber\\
 &&=\gamma
D+(1-\alpha-\gamma)(E-F),\nonumber
\end{eqnarray*}
where
\begin{eqnarray}
D&=&\sum_{\substack{i_1,\ldots,i_k\ge 1\\
   {\rm distinct}}}\int_{\{s\in
\cS^\downarrow:\ s_{i_1}\in A_1,\ldots,s_{i_k}\in
A_k\}}\frac{s_{i_1}\ldots s_{i_k}}{\prod_{j=1}^{k-1}(1-\sum_{l=1}^j
  s_{i_l})}{\rm
PD}^*_{\alpha,-\alpha-\gamma}(ds)\nonumber\\
&=&{\rm GEM}^*_{\alpha,-\alpha-\gamma}(A_1\times\ldots\times A_k),\nonumber\\
 E&=&\sum_{\substack{i_1,\ldots,i_k\ge 1\\
   {\rm distinct}}}\int_{\{s\in
\cS^\downarrow:\ s_{i_1}\in A_1,\ldots,s_{i_k}\in
A_k\}}\left(1-\sum_{u=1}^k
s_{i_u}^2\right)\frac{s_{i_1}\ldots s_{i_k}}{\prod_{j=1}^{k-1}(1-\sum_{l=1}^j
  s_{i_l})}{\rm
PD}^*_{\alpha,-\alpha-\gamma}(ds)\nonumber\\
&=&\int_{A_1\times\ldots\times A_k}\left(1-\sum_{i=1}^k x_i^2\right ){\rm
GEM}^*_{\alpha,-\alpha-\gamma}(dx)\nonumber\\
F&=&\sum_{\substack{i_1,\ldots,i_k\ge 1\\
   {\rm distinct}}}\int_{\{s\in
\cS^\downarrow:\ s_{i_1}\in A_1,\ldots,s_{i_k}\in
A_k\}}\left(\sum_{v\not\in\{i_1,\ldots,i_k\}}
s_v^2\right)\frac{s_{i_1}\ldots s_{i_k}}{\prod_{j=1}^{k-1}(1-\sum_{l=1}^j
  s_{i_l})}{\rm
PD}^*_{\alpha,-\alpha-\gamma}(ds)\nonumber\\
&=&\sum_{\substack{i_1,\ldots,i_{k+1}\ge 1\\
   {\rm distinct}}}\int_{\{s\in
\cS^\downarrow:\ s_{i_1}\in A_1,\ldots,s_{i_k}\in A_k\}}\frac{s_{i_{k+1}}^2}{1-\sum_{l=1}^k
s_{i_l}}
\frac{s_{i_1}\ldots s_{i_{k+1}}}{\prod_{j=1}^k(1-\sum_{l=1}^j
  s_{i_l})}{\rm
PD}^*_{\alpha,-\alpha-\gamma}(ds)\nonumber\\
&=&\int_{A_1\times\ldots\times A_k\times[0,1]}\frac{x_{k+1}}{1-\sum_{i=1}^k x_i}{\rm
GEM}^*_{\alpha,-\alpha-\gamma}(d(x_1,\ldots ,x_{k+1})).\nonumber
\end{eqnarray}
Applying (\ref{GEM}) to $F$ (and setting $\theta=-\alpha-\gamma$),
then integrating out $x_{k+1}$, we get:
$$F=\int_{A_1\times\ldots\times A_k}\frac{1-\alpha}{1+(k-1)\alpha-\gamma}\left(1-\sum_{i=1}^k
x_i\right)^2{\rm GEM}^*_{\alpha,-\alpha-\gamma}(dx).$$ Summing over $D,
E, F$, we obtain the formula stated in Proposition \ref{prop4}.
\end{proof}

As the model related to stable trees is a special case of the
alpha-gamma model when $\gamma=1-\alpha$, the sized-biased
dislocation measure for it is
$$\nu^{\rm sb}_{\alpha,1-\alpha}(ds)=\gamma{\rm GEM}^*_{\alpha,-1}(ds).$$

For general $(\alpha,\gamma)$, the explicit form of the dislocation 
measure in size-biased order, specifically the density $g_{\alpha,\gamma}$ of the first marginal of $\nu^{\rm sb}_{\alpha,\gamma}$,
yields immediately the tagged particle \cite{Ber-hom} L\'{e}vy measure associated with a fragmentation process with alpha-gamma 
dislocation measure. 

\begin{coro}\label{cortps} Let $(\Pi^{\alpha,\gamma}(t),t\geq 0)$ be an exchangeable homogeneous 
$\cP_{\bN}$-valued fragmentation process with dislocation measure $\nu_{\alpha,\gamma}$. Then, for the
size $|\Pi_{(i)}^{\alpha,\gamma}(t)|$ of the block containing $i\ge 1$, the process
$\xi_{(i)}(t)=-\log|\Pi_{(i)}^{\alpha,\gamma}(t)|$, $t\geq 0$, is a pure-jump
subordinator with L\'{e}vy measure 
\begin{eqnarray*}
   \Lambda_{\alpha,\gamma}(dx)\;=\;e^{-x}g_{\alpha,\gamma}(e^{-x})dx\!\!&=&\!\!\frac{\alpha\Gamma(1-\gamma/\alpha)}{\Gamma(1-\alpha)\Gamma(1-\gamma)}\left(1-e^{-x}\right)^{-1-\gamma}\left(e^{-x}\right)^{1-\alpha}\\
              &&\times \left(\gamma+(1-\alpha-\gamma)\left(2e^{-x}(1-e^{-x})+\frac{\alpha-\gamma}{1-\gamma}(1-e^{-x})^2\right)\right)dx.
\end{eqnarray*}
\end{coro}

\subsection{Convergence of alpha-gamma trees to self-similar CRTs}\label{sechmpw}
In this subsection, we will prove that the delabelled alpha-gamma
trees $T_n^\circ$, represented as $\bR$-trees with unit edge lengths and suitably rescaled 
converge to CRTs as $n$ tends to infinity.

\begin{lemm}
If $(\widetilde{T}_n^\circ)_{n\geq 1}$ are strongly sampling consistent discrete
fragmentation trees associated with dislocation measure
$\nu_{\alpha,-\alpha-\gamma}$, then
$$
    \frac{\widetilde{T}_n^\circ}{n^\gamma}\rightarrow\mathcal{T}^{\alpha,\gamma}$$
in the Gromov-Hausdorff sense, in probability as $n\rightarrow
\infty$.
\end{lemm}
\begin{proof} Theorem 2 in
\cite{HMPW} says that a strongly sampling consistent family of discrete
fragmentation trees $(\widetilde{T}_n^\circ)_{n\geq1}$ converges in probability to a CRT
$$\frac{\widetilde{T}_n^\circ}{n^{\gamma_{\nu}}\ell(n)\Gamma(1-\gamma_{\nu})}\rightarrow
\mathcal{T}_{(\gamma_{\nu},\nu)}$$ for the Gromov-Hausdorff metric
if the dislocation measure $\nu$ satisfies
following two conditions:
\begin{equation}
          \nu(s_1\leq1-\varepsilon)=\varepsilon^{-\gamma_\nu}\ell(1/\varepsilon);\label{v1}
       \end{equation}
\begin{equation}
          \int_{\cS^\downarrow}\sum_{i\geq2}s_i|\rm ln\it s_i|^\rho\nu(ds)<\infty,\label{v2}
        \end{equation}
where $\rho$ is some positive real number, $\gamma_{\nu}\in(0,1)$,
and $x\mapsto\ell(x)$ is slowly varying as $x\rightarrow\infty$.

 By virtue of (19) in \cite{HMPW}, we know that (\ref{v1}) is
equivalent to $$
     \Lambda([x,\infty))= x^{-\gamma_\nu}\ell(1/x),\qquad\mbox{as $x\downarrow0$,}$$ 
where $\Lambda$ is the L\'{e}vy measure of the tagged particle subordinator as in Corollary \ref{cortps}.
So, the dislocation measure $\nu_{\alpha,\gamma}$ satisfies
(\ref{v1}) with $\ell(x)\rightarrow \gamma\alpha
\Gamma(1-\gamma/\alpha)/\Gamma(1-\alpha)\Gamma(2-\gamma)$ and
$\gamma_{\nu_{\alpha,\gamma}}=\gamma$. Notice that
$$\int_{\cS^\downarrow}\sum_{i\geq2}s_i|{\rm ln} s_i|^\rho\nu_{\alpha,\gamma}(ds)\leq\int_0^\infty
x^\rho\Lambda_{\alpha,\gamma}(dx).$$ As
$x\rightarrow\infty$, $\Lambda_{\alpha,\gamma}$ decays
exponentially, so $\nu_{\alpha,\gamma}$ satisfies condition
(\ref{v2}). This completes the proof.
\end{proof}

\begin{proof}[Proof of Corollary \ref{dconv}] The splitting rules of
 $T_n^\circ$ are the same as those of $\widetilde{T}_n^\circ$, which leads to the
identity in distribution for the whole trees. The preceding lemma yields convergence in 
distribution for $T_n^\circ$. \end{proof}


\section{Limiting results for labelled alpha-gamma trees}

In this section we suppose $0<\alpha<1$ and $0<\gamma\le\alpha$. In the boundary case $\gamma=0$ trees grow logarithmically and
do not possess non-degenerate scaling limits; for $\alpha=1$ the study in Section \ref{sectalpha1} can be refined to give
results analogous to the ones below, but with degenerate tree shapes.

\subsection{The scaling limits of reduced alpha-gamma trees}
 For $\tau$ a rooted $\mathbb{R}$-tree and
$x_1,\ldots ,x_n\in\tau$, we call $R(\tau,
x_1,\ldots ,x_n)=\bigcup_{i=1}^n[[\rho,x_i]]$ the reduced subtree
associated with $\tau,x_1,\ldots ,x_n$, where $\rho$ is the root of
$\tau$.

As a fragmentation CRT, the limiting CRT
$(\mathcal{T}^{\alpha,\gamma},\mu)$ is
naturally equipped with a mass measure $\mu$ and contains subtrees
$\widetilde{\mathcal{R}}_k,k\geq 1$ spanned by $k$ leaves chosen
independently according to $\mu$. Denote the 
discrete tree without edge lengths by $\tilde{T}_n$ -- 
it has \textit{exchangeable} leaf labels. Then $\widetilde{\mathcal{R}}_n$ is the
almost sure scaling limit of the reduced trees $R(\widetilde{T}_n,[k])$,
by Proposition 7 in \cite{HMPW}.

On the other hand, if we denote by $T_n$ 	the (non-exchangeably) labelled trees obtained
via the alpha-gamma growth rules, the above result will not
apply, but, similarly to the result for the alpha model shown in
Proposition 18 in \cite{HMPW}, we can still establish a.s. convergence of
the reduced subtrees in the alpha-gamma model as stated in Theorem
\ref{LE}, and the convergence result can be strengthened as follows.

\begin{prop}\label{slagt}In the setting of Theorem \ref{LE}
$$
           (n^{-\gamma}R(T_n,[k]), n^{-1}W_{n,k})\rightarrow
           (\cR_k,W_k)\qquad\mbox{a.s. as $n\rightarrow\infty$,}$$
           in the sense of Gromov-Hausdorff convergence, where
           $W_{n,k}$  is the total number of leaves in
           subtrees  of $T_n\backslash
           R(T_n,[k])$ that are linked to the present branch points of
           $R(T_n,[k])$.
           \end{prop}

\begin{proof}[Proof of Theorem \ref{LE} and Proposition \ref{slagt}]
Actually, the labelled discrete tree $R(T_n,[k])$ with edge lengths removed is $T_k$ 
for all $n$. Thus, it suffices to prove the convergence of its total
length and of its edge length proportions. 

Let us consider a first urn model, cf. \cite{Fel1}, where at level $n$ the urn contains a 
black ball for each leaf in a subtree that is directly connected to a branch point of
$R(T_n,[k])$, and a white ball for each 
leaf in one of the remaining subtrees connected to the edges of $R(T_n,[k])$. 
Suppose that the balls are labelled like the leaves they represent. If the urn then
contains $W_{n,k}=m$ white balls and $n-k-m$ black balls, the induced partition of $\{k+1,\ldots ,n\}$
has probability function
$$p(m,
n-k-m)=\frac{\Gamma(n-m-\alpha-w)\Gamma(w+m)\Gamma(k-\alpha)}{\Gamma(k-\alpha-w)\Gamma(w)\Gamma(n-\alpha)}
=\frac{B(n-m-\alpha-w,w+m)}{B(k-\alpha-w,w)}
$$
where $w=k(1-\alpha)+\ell\gamma$ is the total weight on the $k$ leaf edges and $\ell$ other edges of
$T_k$. As $n\rightarrow\infty$, the urn is such that $W_{n,k}/n\rightarrow W_k$ a.s., where $W_k\sim {\rm
beta}((k-1)\alpha-l\gamma,k(1-\alpha)+l\gamma)$.

We will partition the white balls further. Extending the notions of spine, spinal subtrees and spinal bushes 
from Proposition \ref{prop10} ($k=1$), we call, for $k\ge 2$, \em skeleton \em the tree $S(T_n,[k])$ of $T_n$ spanned by the {\sc root} and 
leaves $[k]$ including the degree-2 vertices, for each such degree-2 vertex $v\in S(T_n,[k])$, we consider the skeletal 
subtrees $S^{\rm sk}_{vj}$ that we join together into a \em skeletal bush \em $S^{\rm sk}_v$. 
Note that the total length $L_k^{(n)}$ of the skeleton $S(T_n,[k])$ will increase by 1 if leaf $n+1$ in
$T_{n+1}$ is added to any of the edges of $S(T_n,[k])$; also, $L_k^{(n)}$ is equal to the number of skeletal bushes 
(denoted by $\overline{K}_n$) plus the original total length of $k+\ell$ of $T_k$. Hence, as $n\rightarrow
\infty$
\begin{equation}
\frac{L_k^{(n)}}{n^{\gamma}}\sim\frac{\overline{K}_n}{W_{n,k}^{\gamma}}\left
(\frac{W_{n,k}}{n}\right )^{\gamma}
\sim\frac{\overline{K}_n}{W_{n,k}^{\gamma}}W_k^{\gamma}.\label{sim}
\end{equation}
The partition of leaves (associated with white balls), where each skeletal bushes gives rise to a block,
follows the dynamics of a Chinese Restaurant Process with $(\gamma,w)$-seating plan: 
given that the number of white balls in the first urn is $m$ and that there are $K_m:=\overline{K}_n$
skeletal bushes on the edges of $S(T_n,[k])$ with $n_i$ leaves on the $i$th bush, the next leaf associated 
with a white ball will be inserted into any particular bush with $n_i$ leaves with probability proportional to $n_i-\gamma$ and
will create a new bush with probability proportional to $w+K_m\gamma$. 
Hence, the EPPF of this partition of the white balls is
$$p_{\gamma,w}(n_1,\ldots,n_{K_m})=\frac{\gamma^{K_m-1}\Gamma(K_m+w/\gamma)\Gamma(1+w)}
{\Gamma(1+w/\gamma)\Gamma(m+w)}\prod_{i=1}^{K_m}\Gamma_{\gamma}(n_i).$$
Applying Lemma \ref{crp2} in connection with (\ref{sim}),
we get the probability density of $L_k/W_k^{\gamma}$ as
specified.

Finally, we set up another urn model that is updated whenever a new skeletal bush is created. This model records
the edge lengths of $R(T_n,[k])$. The alpha-gamma growth rules assign weights $1-\alpha+(n_i-1)\gamma$ to leaf
edges of $R(T_n,[k])$ and weights $n_i\gamma$ to other edges of length $n_i$, and each new skeletal bush makes one of the weights
increase by $\gamma$. Hence, the conditional
probability that the length of each edge is $(n_1,\ldots ,n_{k+l})$ at
stage $n$ is that
$$
\frac{\prod_{i=1}^{k}\Gamma_{1-\alpha}(n_i)\prod_{i=k+1}^{k+\ell}\Gamma_{\gamma}(n_i)}{\Gamma_{k\alpha+\ell\gamma}(n-k)}.$$
Then $D_k^{(n)}$ converge a.s. to the Dirichlet limit as specified.
Moreover, $L_k^{(n)}D_k^{(n)}\rightarrow L_kD_k$ a.s., and it is easily seen that this implies
convergence in the Gromov-Hausdorff sense.

The above argument actually gives us the conditional distribution of
$L_k/W_k^{\gamma}$ given $T_k$ and $W_k$, which does not
depend on $W_k$. Similarly, the conditional distribution of
$D_k$ given given $T_k$, $W_k$ and $L_k$ does not depend
on $W_k$ and $L_k$. Hence, the conditional independence of $W_k$,
$L_k/W_k^{\gamma}$ and $D_k$ given $T_k$ follows.
\end{proof}

\subsection{Further limiting results}\label{secbw}

Alpha-gamma trees not only have edge weights but also vertex
weights, and the latter are in correspondence with the vertex degrees. We can get a result 
on the limiting ratio between the
degree of each vertex and the total number of leaves.

\begin{prop}
  Let $(c_1+1,\ldots,c_{\ell}+1)$ be the degree of each vertex in $T_k$, listed by
  depth first search. The ratio between the degrees in $T_n$ of these vertices
  and $n^\alpha$ will
  converge to
  $$C_k=(C_{k,1},\ldots,C_{k,\ell})=
  \overline{W}_k^\alpha M_k D_k^\prime,\qquad
		\mbox{where $D_k^\prime\sim{\rm Dirichlet}(c_1-1-\gamma/\alpha,\ldots,c_\ell-1-\gamma/\alpha)$}$$
  and $M_k$ are conditionally independent of $W_k$ given $T_k$, where $\overline{W}_k=1-W_k$, and $M_k$ has density
  $$\frac{\Gamma(\overline{w}+1)}{\Gamma(\overline{w}/\alpha+1)}s^{\overline{w}/\alpha}g_\alpha(s),\qquad s\in(0,\infty),$$
   $\overline{w}=(k-1)\alpha-\ell\gamma$ is total branch point weight in $T_k$ and $g_\alpha(s)$ is the Mittag-Leffler density.
\end{prop}
\begin{proof} Recall the first urn model in the preceding proof which assigns
  colour black to leaves attached in subtrees of branch points of $T_k$. We will
  partition the black balls further. 
  The partition of leaves (associated with black balls), where each \em subtree \em $S^{\rm sk}_{vj}$ of a branch point  
  $v\in R(T_n,[k])$ gives rise to a block, follows the dynamics of a Chinese Restaurant Process with $(\alpha,\overline{w})$-seating 
  plan. Hence, the total degree $C^{\rm tot}_k(n)/\overline{W}_{n,k}^\alpha\rightarrow M_k$ a.s., where $C_k^{\rm tot}(n)$ is
  the sum of degrees in $T_n$ of the branch points of $T_k$, and $\overline{W}_{n,k}=n-k-W_{n,k}$ is the total number of leaves of
  $T_n$ that are in subtrees directly connected to the branch points of $T_k$. 

  Similarly to the discussion of edge
  length proportions, we now see that the sequence of degree proportions will converge a.s. to the Dirichlet limit as specified. 
  Since $1-W_k$ is the a.s. limiting proportion of leaves in subtrees connected to the vertices of $T_k$.
\end{proof}

Given an alpha-gamma tree $T_n$, if we decompose along the spine that
connects the {\sc root} to leaf 1, we will find the leaf numbers of
subtrees connected to the spine is a Chinese restaurant partition of
$\{2,\ldots,n\}$ with parameters $(\alpha, 1-\alpha)$. Applying Lemma
\ref{crp1}, we get following result.

\begin{prop} Let $(T_n,n\ge 1)$ be alpha-gamma trees.
  Denote by $(P_1,P_2,\ldots)$ the limiting frequencies of the leaf numbers of each subtree
  of the spine connecting the {\sc root} to leaf 1 in the order of 
  appearance. These can be represented as
  $$(P_1,P_2,\ldots)=(W_1,\overline{W}_1W_2,\overline{W}_1\overline{W}_2W_3,\ldots )$$
  where the $W_i$ are independent, $W_i$ has ${\rm beta}(1-\alpha,
  1+(i-1)\alpha)$ distribution, and $\overline{W}_i=1-W_i$.
\end{prop}

Observe that this result does not depend on $\gamma$. This observation also follows from Proposition \ref{prop6}, because colouring
(iv)$^{\rm col}$ and crushing (cr) do not affect the partition of leaf labels according to subtrees of the spine.

\bibliographystyle{abbrv}
\bibliography{cfw}

\end{document}